\newif\ifsubsections
\subsectionstrue 

\documentclass[]{pcmi}

\usepackage[sc]{mathpazo}          
\usepackage{eulervm}               
\usepackage[scaled=0.86]{berasans} 
\usepackage[scaled=1]{inconsolata} 
\usepackage[T1]{fontenc}

\usepackage[%
	protrusion=true,
	expansion=false,
	auto=false
	]{microtype}

\usepackage{xcolor}
\usepackage{graphicx}
\graphicspath{{./figures/}}

\ifdraft
	\definecolor{linkred}{rgb}{0.7,0.2,0.2}
	\definecolor{linkblue}{rgb}{0,0.2,0.6}
\else
	\definecolor{linkred}{rgb}{0.0,0.0,0.0}
	\definecolor{linkblue}{rgb}{0,0.0,0.0}
\fi

\usepackage[all]{xy}

\PassOptionsToPackage{hyphens}{url} 
\usepackage[
    setpagesize=false,
    pagebackref,
	pdfpagelabels=false,
    pdfstartview={FitH 1000},
    bookmarksnumbered=false,
    linktoc=all,
    colorlinks=true,
    anchorcolor=black,
    menucolor=black,
    runcolor=black,
    filecolor=black,
    linkcolor=linkblue,
	citecolor=linkblue,
	urlcolor=linkred,
]{hyperref}
\usepackage[backrefs,msc-links,nobysame]{amsrefs}

\customizeamsrefs 

\newtheorem{tm}[equation]{Theorem}

\newtheorem{rmk}[equation]{Remark}
\newtheorem{cor}[equation]{Corollary}
\newtheorem{ex}[equation]{Example}
\newtheorem{exe}[equation]{Exercise}
\newtheorem{fact}[equation]{Fact}
\newtheorem{??}[equation]{Question}
\newtheorem{defi}[equation]{Definition}

\newcommand{\ben}{\begin{enumerate}}
\newcommand{\een}{\end{enumerate}}
\newcommand{\bit}{\begin{itemize}}
\newcommand{\eit}{\end{itemize}}
\newcommand{\beq}{\begin{equation}}
\newcommand{\eeq}{\end{equation}}
\newcommand{\beqstar}{\begin{equation*}}

\newcommand{\la}{\label}
\newcommand{\n}{\noindent}
\newcommand\ci{\cite}

\newcommand{\im}{ \hbox{\rm Im} }
\newcommand{\ke}{ \hbox{\rm Ker} }
\newcommand{\lorw}{\longrightarrow}

\commentout{
\newcommand\rat{{\Bbb Q}}

\newcommand\comp{{\Bbb C}}

\newcommand\zed{{\Bbb Z}}

\newcommand\pn[1]{{\Bbb P}^{#1}}

}
\newcommand\rat{\mathbb{Q}}

\newcommand\comp{\mathbb{C}}

\newcommand\zed{\mathbb{Z}}

\newcommand\pn[1]{\mathbb{P}^{#1}}

\def\l{\lambda}

\newcommand\e{\epsilon}

\newcommand{\w}[1]{\widetilde{#1}}

\newcommand{\ov}[1]{\overline{#1}}

\commentout{
\newcommand{\ptd}[1]{ \,^{\frak p}\!\tau_{ \leq {#1} } }

\newcommand{\td}[1]{ \tau_{ \leq {#1} } }
\newcommand{\tu}[1]{ \tau_{ \geq {#1} } }

\newcommand{\pc}[2]{ \,^{\frak p}\!{\mathcal H}^{#1}({#2})   }
\newcommand{\pcs}{ \,^{\frak p}\!{\mathcal H}   }

}

\newcommand{\ptd}[1]{ \,^\mathfrak{p}\!\tau_{ \leq {#1} } }

\newcommand{\td}[1]{ \tau_{ \leq {#1} } }
\newcommand{\tu}[1]{ \tau_{ \geq {#1} } }

\newcommand{\pc}[2]{ \,^\mathfrak{p}\!{\mathcal H}^{#1}({#2})   }
\newcommand{\pcs}{ \,^\mathfrak{p}\!{\mathcal H}   }

\newcommand{\sq}{\m{R}^q}

\newcommand{\pev}{\m{EV}}

\newcommand{\m}[1]{\mathcal{#1}}
\newcommand{\ms}[1]{\mathcal{#1}}
\newcommand{\bb}[1]{\mathbb{#1}}

\begin{document}

%
%
%
%
%
%

\title{Perverse sheaves and the topology of algebraic varieties}

%
%
\author{Mark Andrea A. de Cataldo}
\address{Department of Mathematics,
SUNY at Stony Brook,
Stony Brook,  NY 11794, USA}
\email{mark.decataldo@stonybrook.edu}

\thanks{
Partially supported by N.S.F. grant DMS-1301761 and by a grant from the Simons Foundation (\#296737 to Mark de Cataldo)}

\dedicatory{\large Dedicato a Mikki}

\date{July 2, 2015}

%
%

%
%
\maketitle

%
%

\tableofcontents

\providecommand{\bysame}{\leavevmode \hbox \o3em
{\hrulefill}\thinspace}

\textbf{Goal of the lectures.}
The goal of these lectures is to introduce the novice to the use of perverse sheaves in complex algebraic geometry
and to what is perhaps  the deepest known fact relating the  homological/topological invariants of  the source and target of 
a proper map of complex algebraic varieties, 
namely the  decomposition theorem.

\textbf{Notation.}
A variety is a  complex algebraic variety, which we do not assume to be irreducible, nor reduced. We work with cohomology with $\rat$-coefficients. $\zed$-coefficients do not fit well in our story. As we rarely focus on a single cohomological degree,
for the most part we consider the total, graded cohomology groups, e.g. $H^*(X,\rat)$.

\textbf{Bibliographical references.} The main reference is  the survey \ci{bams}  and the extensive bibliography contained in it,
most of which is not be reproduced here. 
This allowed me to try to minimize the continuous  distractions  related to  the peeling apart of  the various versions of the results and of the attributions. The reader may also consult the discussions in  \ci{bamsv1} that did not make it into the very different final version \ci{bams}. 

\textbf{Style of the lectures and of the lecture notes.} I hope to deliver my lectures in a rather informal style.
I plan to introduce some main ideas, followed by what I believe to be a striking application, often with an idea of proof.
The lecture notes are not intended to replace in any way the existing literature on the subject, they are a mere amplification
of what I can possibly touch upon  during the five one-hour lectures.
As it is usual when meeting a new concept, the theorems and the applications are very important, but I also believe that working with  examples,
no matter how lowly  they may seem,
can be truly illuminating and useful in building one own's local and global picture.  Because of the time factor, I cannot possibly fit
many of these examples in the flow of the  lectures. This is why there are plenty of exercises, which are not just about examples, but at time deal head-on  with actual important theorems. I could have laid-out several more exercises
(you can look at my lecture notes  \ci{trieste}, or at my little book \ci{decbook} for more exercises), but
I tried to choose ones that would complement well the lectures; too much of anything is  not a good thing anyway.

\textbf{What is missing from these lectures?} A lot! Two related topics come to mind: vanishing/nearby cycles and constructions of perverse sheaves; see the survey above for a quick introduction to both. To compound this infamy, there is no discussion
of the equivariant picture \ci{bl}.

\textbf{An afterthought.} The 2015 PCMI is now over. Even though I have been away from Mikki, Caterina, Amelie (Amie!) and Dylan
for three weeks, my PCMI experience has been wonderful. If you love math, then you should consider participating in future PCMIs.
Now, let us get to Lecture 1.

\section{Lecture 1: The decomposition theorem}\la{lz1}

\textbf{Summary of Lecture 1.}  {\em Deligne theorem on the degeneration of the Leray spectral sequence for smooth projective maps;
this is the 1968 prototype of the 1982 decomposition theorem.
Application, via the use of the theory of mixed Hodge structures, to the global invariant cycle theorem, a remarkable
topological property enjoyed by families of projective manifolds and  compactifications of their total spaces. The main theorem
of these lectures, the decomposition theorem,
stated in cohomology. Application to a proof of the local invariant cycle theorem, another remarkable  topological property
concerning the degenerations of families of projective manifolds. Deligne theorem, including semisimplicity
of the direct image sheaves, in the derived category. The decomposition theorem: the direct image complex splits in the derived category into a direct sum of shifted and twisted intersection complexes supported on the target of a proper map.}

\subsection{Deligne theorem in cohomology}\la{lz1dt}
 {\bf Warm-up: the K\"unneth formula and a question.} Let $Y,F$ be  varieties.
Then 
\beq\la{kufo}
H^*(Y\times F,\rat) = \bigoplus_{q\geq 0} H^{*-q}(Y,\rat) \otimes H^q(F,\rat).
\eeq
Note that the restriction map $H^*(Y\times F,\rat) \to H^*(F,\rat)$ is surjective.
\begin{??}\la{gictku} Let $F$ be a  projective manifold and let  $\ov{X}$ be a projective manifold completing $Y\times F$ (i.e. $Y\times F$ Zariski-dense in $\ov{X}$). What can we say
about the restriction map $H^*(\overline{X},\rat) \to H^*(F,\rat)$? 
\end{??}

{\bf Answer:} Theorem~\ref{gict} gives an answer in the  more general (and more interesting) setting of families of projective manifolds
and  the compactifications (in the sense above) of their total spaces.

The decomposition theorem  has an important precursor in Deligne theorem, which can be viewed as the decomposition theorem in the absence of singularities of the domain, of the target
{\em and}  of the map. We start by stating the cohomological version of his theorem.

 \begin{tm}\la{codeth} {\rm ({\bf Blanchard-Deligne 1968 theorem in cohomology} \ci{dess})}
For any smooth  projective map\footnote{Smooth: submersion; projective: factors as $X \to Y\times \bb{P}\to Y$ (closed embedding, projection)} $f: X\to Y$ of algebraic manifolds,  there is an isomorphism 
\beq\la{eq1}
H^*(X,\rat) \cong \bigoplus_{q\geq 0} H^{*-q}(Y, R^q f_*\rat_X).
\eeq
More precisely, the Leray spectral sequence (see \S\ref{exlz1})  of the map $f$ is $E_2$-degenerate.
\end{tm}
\begin{proof} Exercise~\ref{dellef} guides you through Deligne's classical trick (the Deligne-Lefschetz criterion) of using the hard Lefschetz theorem  on the fibers to force the triviality of the differentials of the spectral sequence.
\end{proof}

Compare (\ref{kufo}) and (\ref{eq1}): both present cohomological shifts; both express the cohomology of the l.h.s. via
cohomology groups on $Y$; in the former case, we have cohomology with constant coefficients; in the latter, and this is crucial,
we have cohomology with locally constant coefficients.

Deligne theorem  is central  in the study of the topology of algebraic varieties. Let us discuss one striking application
of this result: the global invariant cycle theorem.

\subsection{The global invariant cycle theorem}\la{er431}
Let $f: X \to Y$ be a smooth  and projective map of algebraic manifolds, let $j:X \to \ov{X}$
be an open immersion into a projective manifold and let $y\in Y.$ What are
the images of $H^*(X,\rat)$ and $H^*(\ov{X},\rat)$ via the restriction maps into $H^*(f^{-1}(y), \rat)$?
The answer is the global invariant cycle Theorem~\ref{gict} below.

The direct image sheaf $\sq :=R^q f_* \rat_X$  on $Y$ is the sheaf associated with the pre-sheaf
\[U \mapsto H^q(f^{-1}(U), \rat). \]
In view of Ehresmann lemma, the proper\footnote{Proper :=  the pre-image of compact is compact; it is the ``relative'' version of compactness} submersion $f$
is a $C^\infty$ fiber bundle. The sheaf $\sq$ is then locally constant with stalk \[\sq_y=H^q(f^{-1}(y), \rat).\]
The fundamental group $\pi_1 (Y,y)$ acts via linear transformations on  $\sq_y$: pick a loop $\gamma (t)$ at $y$
and use a trivialization of the bundle along the loop to move vectors in $\sq_y$ along $\sq_{\gamma (t)}$,
back to $\sq_y$ (monodromy action for the locally constant sheaf~$\sq$). 

The global 
sections of $\sq$ identify with the monodromy invariants $(\sq_y)^{\pi_1} \subseteq  \sq_y.$ Note that this subspace is defined
topologically. The cohomology group $\sq_y= H^q(f^{-1}(y), \rat)$ has it own Hodge $(p,p')$-decomposition (pure Hodge
structure of weight $q$), an
algebro-geomeric structure. 

How is  $(\sq_y)^{\pi_1} \subseteq  \sq_y$ placed  with respect to the  Hodge structure?

The $E_2$-degeneration Theorem~\ref{codeth} yields  the following immediate, yet, remarkable, 
consequence:  \beq\la{grict}
H^q (X,\rat) \stackrel{\rm{surj}}\lorw  (\sq_y)^{\pi_1} \subseteq \sq_y,
\eeq
 i.e.  {\em the restriction map in cohomology, which automatically factors
through the invariants, in fact factors surjectively through them.}

The theory of mixed Hodge structures now tells us that the monodromy invariant subspace
$(\sq_y)^{\pi_1} \subseteq \sq_y$ (a topological gadget) is in fact  a Hodge substructure, i.e. it inherits the Hodge $(p,p')$-decomposition
(the algebro-geometric gadget).  

The same mixed theory implies that  highly non trivial fact (Exercise~\ref{rty60}) that the images of the restriction maps  from $H^*(\ov{X}, \rat)$ and $H^*(X, \rat)$ 
into $H^*(f^{-1}(y), \rat)$ coincide. 

We have reached the following conclusion, proved by Deligne in 1972.

\begin{tm}\la{gict} {\rm {\bf (Global invariant cycle theorem \ci{ho2})}}
Let $f: X \to Y$ be a smooth  and projective map of algebraic manifolds, let $j:X \to \ov{X}$
be an open immersion into a projective manifold and let $y\in Y.$ Then the images 
of $H^*(\ov{X},\rat)$ and $H^*(X, \rat)$ into $H^*(f^{-1}(y),\rat)$ coincide with the subspace of monodromy invariants.
In particular, this latter is  a Hodge substructure of the pure Hodge structure $H^q(f^{-1}(y), \rat)$.
\end{tm}

This theorem provides a far-reaching answer to Question~\ref{gictku}. Note that the Hopf examples in Exercise~\ref{hopft}
show that such a nice  general answer is not possible outside of the realm of complex algebraic geometry: there are two obstacles,
i.e.  the non $E_2$-degeneration, and  the absence
of the special kind of global constraints imposed by mixed Hodge structures.

\subsection{Cohomological decomposition theorem}\la{thedtc}$\;$
The decomposition theorem   is a generalization  of Deligne's  Theorem~\ref{codeth} for smooth proper
maps to the case of arbitrary  proper maps of algebraic varieties: compare (\ref{eq1}) and (\ref{3edd}).  It was first proved by
Beilinson-Bernstein-Deligne-Gabber in their monograph \ci{bbd} (Th\'eor\`em 6.2.5) on perverse sheaves.

A possible initial psychological drawback, when compared with Deligne's theorem, is that 
even if one insists in  dealing with maps of  projective manifolds, the statement is not about
cohomology with locally constant coefficients, but it  requires the  Goresky-MacPherson intersection cohomology groups
with twisted coefficients on various subvarieties of the target of the map. However, this is precisely why this theorem is so striking!

To get to the point, 
for now we simply say that  we have  the intersection 
cohomology groups $IH^*(S,\rat)$  of an irreducible variety $S$; they agree with the ordinary
cohomology groups when $S$ is nonsingular.  The theory is very flexible
as it allows for twisted coefficients: given a locally constant sheaf $L$ on a dense open subvariety $S^o \subseteq
S_{reg} \subseteq  S,$ we get the intersection  cohomology groups $IH^*(S,L)$  of $S$.
We may call such pairs $(S,L)$, enriched varieties (see \ci{macicm}, p.222); this explains the notation $\pev$ below.


\begin{tm}\la{dt00} {\rm ({\bf Cohomological decomposition theorem})}
Let $f: X \to Y$ be a proper map of complex algebraic varieties. For every $q \geq 0$, there is a finite collection $\pev_q$ of pairs
 $(S,L)$ with $S\subseteq Y$ pairwise distinct  closed subvarieties of Y,  and an isomorphism
\beq\la{3edd}
IH^*(X,\rat) \cong \oplus_{q\geq 0,\pev_q} IH^{*-q}(S,L).
\eeq
\end{tm}
Note that the same $S$ could appear for distinct $q$'s.

Deligne theorem in cohomology is a special case. In particular, we can deduce an appropriate version
of the global invariant cycle theorem \ci{bbd}, 6.2.8. Let us instead focus on its local counterpart.

\subsection{The local invariant cycle theorem}\la{pqpw3}
 The decomposition  theorem (\ref{3edd}) has a local flavor over the
target $Y$, in both the Zariski and in the classical topology: replace $Y$ by an open set $U\subseteq Y$, $X$ by $f^{-1}(U)$,
and $S$ by $S \cap U$.

 Let us focus on the classical topology. Let $X$ be nonsingular; this is for the sake of our discussion, for then 
 $IH^*(X,\rat)= H^*(X,\rat)$.
 
Let $y\in Y$ be a point and let us pick a small Euclidean ``ball'' $B_y \subseteq Y$ centered at $y$, 
so that   (\ref{3edd}) reads: 
\[H^*(f^{-1}(y),\rat)=H^* (f^{-1}(B_y),\rat) = \oplus_{q,\pev_q} IH^{*-q} (S\cap B_y, L).\]
(The first ``$=$'' above is non-trivial, as it follows from the constructibility of the direct image complex $Rf_*\rat_X$, 
so that the second term can be identified with the stalk $(R^*f_*\rat_X)_y$, and by the proper base change theorem, that ensures
that this latter is the first term; see Fact~\ref{moimp}.)
Let $f$ be surjective. Let $f^o: X^o \to Y^o$ the restriction of the map $f$ over the open subvariety of $Y$
of  regular
values for $f.$ Let $y^o \in B_y$ be a regular value for $f$.

By looking at Deligne theorem for the map $f^o$ it seems reasonable to expect
that for every $q$ one of the summands in (\ref{3edd})  should be $IH^{*-q}(Y, L_q)$, where $L_q$
is the locally constant $R^qf^o_*\rat$.  This is indeed the case.

If follows that for every $q\geq 0$, we have that $IH^0(B_y, {L_q}_{|Y^o\cap B_y})$ is a direct summand of $H^q(f^{-1}(y),\rat)$,
let us even say that the latter surjects onto the former. Note that we did not assume that $y\in Y^o$. 

The intersection  cohomology group $IH^0(Y, L_q)$ is the space of monodromy invariants
for the representation $\pi_1(Y^o\cap B_y,y^o) \to GL (H^q(f^{-1}(y^o),\rat)$.  Abbreviate the
fundamental group notation to $\pi_{1,loc}$.

We have reached a very important conclusion: 

\begin{tm}\la{lict} {\rm ({\bf Local invariant cycle theorem}, \ci{clem} and \ci{bbd}, 6.2.9)}
Let $f: X \to Y$ be a proper surjective map of algebraic varieties with $X$ nonsingular. Let $y\in Y$ be any point, let $B_y$
be a small Euclidean ball on $Y$ at $y$,  let $y^o \in B_y$ be a regular value of $f.$
Then $H^*(f^{-1}(y),\rat)=H^*(f^{-1}(B_y),\rat)$ surjects onto the local monodromy invariants $H^*(f^{-1}(y^o),\rat)^{\pi_{1,loc}}.$ 
\end{tm}

\subsection{Deligne theorem}\la{dtder}
In fact, Deligne proved something stronger than his cohomological theorem (\ref{eq1}), he proved a decomposition
theorem for the derived direct image under a smooth proper map.

{\bf Pre-warm-up: cohomological shifts.} Given a $\zed$-graded object $K=\oplus_{i \in \zed}K^i$, like the total cohomology
of a variety,  or a complex (of sheaves, for example) on it,  or the total cohomology of such a complex, etc.,  and given an integer $a \in \zed$, we can shift
by the amount $a$  and get a new graded object (with $K^{i+a}$ in degree~$i$)
\beq\la{shifts}
K[a] := \bigoplus_{i\in \zed} K^{i+a}.
\eeq
If $a>0,$ then the effect of this operation is to  ``shift $K$ back by $a$ units.''  Again, if $K$ has non zero entries contained in an interval
$[m,n]$, then $K[a]$ has  non zero entries contained in $[m-a,n-a]$. We have the following basic relation, e.g. for complexes
of sheaves
\[\m{H}^i(K[a])  = \m{H}^{i+a}(K).\] A sheaf $F$ can be viewed as a complex placed in cohomological degree zero;
we can then take  the $F[a]$'s. We can take  a collection of $F_q$'s and form $\oplus_q F_q[-q]$,  which is a complex with {\em trivial differentials.}
Then \[
H^*(Y,\oplus_q F_q[-q])=\oplus_q H^{*-q}(Y, F_q).\]

{\bf Warm-up: K\"unneth for the derived direct image.} Let $f: X:=Y\times F \to Y$ be the projection. Then
there is a canonical isomorphism 
\beq\la{ku12}
Rf_*\rat_X = \oplus_{q\geq 0} \underline{H}^q(F)  [-q]
\eeq
where $\underline{H}^q(F)$ is the constant sheaf on $Y$ with stalk $H^q(F,\rat)$.
The isomorphism takes place in the derived category of the category of sheaves of rational vector spaces on $Y$; this is where  we find the direct image complex $Rf_* \rat_X$, whose cohomology
is  the cohomology of $X$: $H^*(Y, Rf_* \rat_X) = H^*(X,\rat)$.  Exercise~\ref{kunneth} asks you to prove (\ref{ku12}).

Now to Deligne's 1968  theorem.
\begin{tm}\la{detm} {\rm ({\bf Deligne 1968 theorem \ci{dess}; semisimplicity in 1972 \ci{ho2}, \S4.2)} } Let $f: X\to Y$ be a smooth  proper map of algebraic varieties.  ``The derived image complex has trivial differentials'', more precisely, there
is an isomorphism in the derived category
\beq\la{eq1a}
Rf_* \rat_X \cong  \bigoplus_{q\geq 0} R^qf_* \rat[-q].
\eeq
Moreover, the locally constant  direct image sheaves $R^qf_* \rat_X$ are semisimple.
\end{tm}

Theorem~\ref{detm}(\ref{eq1a}), which is proved by means of an $E_2$-degeneration argument\footnote{People refer to it as the 
Deligne-Lefschetz criterion} along the lines
of the one in Exercise~\ref{dellef},  is the ``derived'' version of
  (\ref{eq1}), which follows   by taking cohomology on both sides
of (\ref{eq1a}). In addition to \ci{dess}, you may want to consult the first two pages of
\ci{shockwave}. The semisimplicity result is one of the many amazing applications of the theory of weights (Hodge-theoretic, or Frobenius).

{\bf Terminology and facts about semisimple locally constant sheaves.} To give a locally constant sheaf on $Y$ is the same as giving a representation
of the fundamental group of $Y$ (Exercise~\ref{098}). By borrowing from  the language of representations, we have the notions of
simple (no non trivial locally constant subsheaf; a.k.a. irreducible) and semisimple (direct sum of simples; a.k.a.
completely reducible), indecomposable (no non trivial direct sum decomposition) locally constant sheaves.

Once one has semisimplicity, one can decompose further.
For a semisimple locally constant sheaf $L$, we have the canonical isotypical direct sum decomposition
\beq\la{isotyp}
L = \oplus_\chi L_\chi,
\eeq where each summand is the span of all
mutually isomorphic simple subobjects, and the direct sum ranges over the set of isomorphism classes of irreducible representations
of the fundamental group. In particular, in (\ref{eq1a}), we have $R^qf_* \rat_X =\sq = \oplus_\chi \sq_\chi.$

{\bf What is semisimplicity good for?} Here is the beginning of an answer:
look at Exercise~\ref{hl}, where it is put to good use to give Deligne's proof in \ci{weil2} of the Hard Lefschetz theorem.
In the context of the decomposition theorem, the semisimplicity of the perverse direct images is an essential
ingredient in the proof of the relative hard Lefschetz theorem; see \ci{bbd} and \ci{decmightam} (especially, \S5.1 and \S6.4).

\subsection{The decomposition theorem}\la{thedt}
As we have seen,  Deligne theorem in cohomology has a counterpart in the derived category.
The cohomological decomposition Theorem~\ref{dt00}  also has a stronger counterpart in the derived category,
i.e Theorem~\ref{dt01}.

In these lectures,
 we adopt   a  version of the decomposition theorem that is more general, and simpler to state!, than the one  in \ci{bbd}, 6.2.5 (coefficients of geometric origin)
  and of \ci{samhm} (coefficients in polarizable variations of pure Hodge structures). The version we adopt
is due essentially  to T. Mochizuki \ci{mochizuki}  (with important contributions of C. Sabbah \ci{sabbah})
 and it involves  semisimple coefficients. \ci{mochizuki} works in the context of projective maps of quasi-projective varieties
 and with $\comp$-coefficients; one needs a little bit of tinkering to reach the same conclusions
 for proper maps of complex varieties with  $\rat$-coefficients (to my knowledge, this is not in the literature).

{\bf Warning: $\m{IC}$ vs. $IC$.} We are about to meet the main protagonists of our lectures, the intersection complexes $\m{IC}_S(L)$
with twisted coefficients; in fact,
the actual protagonists are the shifted (see (\ref{shifts}) for the notion of shift):
\beq\la{icvsic}
IC_S(L):= \m{IC}_S(L) [\dim{S}],
\eeq which are perverse sheaves on $S$ and on any 
variety $Y$ for which $S\subseteq Y$ is closed. While $\m{IC}_S(L)$ has non-trivial cohomology sheaves only in the interval
$[0, \dim{S}-1]$, the analogous interval for  $IC_S(L)$ is $[-\dim{S}, -1]$.  Instead of discussing the pro and cons of either notation, let us move on.

{\bf Brief on intersection complexes.}
The intersection cohomology groups of an enriched variety $(S,L)$ are in fact the cohomology groups
of $S$ with coefficients in a very special complex of sheaves called the intersection complex of $S$ with coefficients in $L$
and denoted by $\m{IC}_S(L)$: we have $IH^*(S,L)=H^*(S, \m{IC}_S(L))$.  If $S$ is nonsingular, and $L$ is constant of rank one,
then $\m{IC}_S=\m{IC}_S (\rat) =\rat_S.$ The decomposition theorem in cohomology
(\ref{3edd}) is the shadow in cohomology of a decomposition of the direct image complex $Rf_* \m{IC}_X$
in the derived category of sheaves of rational vector spaces on $Y$. In fact, the dt holds
in the greater generality of  semisimple coefficients.

\begin{tm}\la{dt01} {\rm ({\bf Decomposition theorem})}
Let $f: X \to Y$ be a proper map of complex algebraic varieties. Let $\m{IC}_X(M)$ be the  intersection complex
of $X$ with semisimple twisted coefficients $M$.  For every $q \geq 0$, there is a finite collection $\pev_q$ of pairs
 $(S,L)$ with $S$ pairwise distinct\footnote{The same $S$ could appear for distinct $q$'s} and  $L$ semisimple, and an isomorphism
\beq\la{eq2}
\xymatrix{
Rf_* \m{IC}_X (M) \cong \bigoplus_{q\geq 0, \pev_q} \m{IC}_S (L) [-q].}
\eeq
In particular, by taking cohomology:
\beq\la{3ed}
IH^*(X,M) \cong \bigoplus_{q,\pev_q} IH^{*-q}(S,L).
\eeq
\end{tm}
We have the isotypical decompositions (\ref{isotyp}), which can be plugged into  what above.

\begin{rmk}\la{noz}
{\rm 
The fact that there may be  summands associated with $S \neq Y$ should not come as a surprise. It is a natural fact due to the singularities 
(deviation from being smooth) of the map $f$. One does not need the decomposition theorem 
to get convinced: the reader can work out the case of the blowing up
of the affine plane at the origin; see  also Exercise~\ref{blowups}. In general, it is difficult to predict
which $S$ will appear in  the decomposition theorem; see parts 5 and 7 of Exercise~\ref{g0g0}.
}
\end{rmk}

\subsection{Exercises for Lecture 1}\la{exlz1}

\begin{exe}\la{loccst} {\rm ({\bf Ehresman lemma and local constancy of higher direct images
for proper submersions})
Let $f: X \to Y$ be a map of varieties and 
recall that  $q$-th direct image sheaf $\sq:= R^qf_* \rat_X$ is defined to be the sheafification of the presheaf $Y \supseteq U \mapsto H^q(f^{-1}(U), \rat)$.  If $f$ admits the structure of a $C^\infty$ fiber bundle, then the sheaves $\sq$ are locally constant,
with stalks the cohomology of the fibers. Give examples of maps where this last statement fails (hint: they cannot be proper).
If $f$ is a proper smooth map of complex algebraic varieties, then it admits a structure  of $C^\infty$ fiber bundle (Ehresmann lemma).
Deduce that nonsingular hypersurfaces of fixed degree in complex projective space are all diffeomorphic to each other.
Is the same true in real projective space? Why?
}
\end{exe}

{\bf Quick review of the Leray spectral sequence} (see Grothendieck's  gem ``Tohoku'').
The Leray spectral sequence
for a map $f:X\to Y$ (and for the sheaf $\rat_X$) is a gadget
denoted $E_2^{pq} = H^p(Y, R^qf_* \rat_X) \Rightarrow H^{p+q}(X, \rat).$ There are the natural differentials $d_r: E_r^{pq} \to E_r^{p+r, q-r+1}$,  $d_r^2=0,$
with $r \geq 2$ and $E_{r+1} = H^*(E_r,d_r)$. $E_2$-degeneration means that $d_r=0$ for every $r\geq 2$, so that one has
a cohomological  decomposition  $H^*(X,\rat) \cong \oplus_{q\geq 0} H^{*-q}(Y, R^qf_* \rat_X)$. Note that with $\zed$~coefficients,
$E_2$-degeneration does not imply the existence of an analogous splitting.

\begin{exe}\la{hopft}{\rm  ({\bf Maps of Hopf-type})
Let $a:\comp^2 \setminus o \to \pn{1}\cong S^2$ be the usual map $(x,y) \mapsto (x:y)$.
It induces two more maps, $b: S^3 \to S^2$ and $c: HS:=(\comp^2 \setminus o)/\zed
\to \pn{1}$ (where $1 \in \zed$ acts as multiplication by two). These three maps are
fiber bundles. Show that there cannot be a cohomological decomposition  as in (\ref{eq1}).
Deduce that their  Leray spectral sequence are   not
$E_2$-degenerate. Observe that the conclusion of the global invariant cycle theorem concerning the surjectivity
onto the monodromy invariants  fails  in all three cases.
}
\end{exe}

\begin{exe}\la{dellef}
{\rm  ({\bf Proof of the cohomological decomposition (\ref{eq1}) via hard Lefschetz})
Let us recall the hard Lefschetz theorem: let $X$ be a projective manifold  of dimension $d$, and let $\eta \in H^2(X,\rat)$
be the first Chern class of an ample line bundle on $X;$ then for every $q \geq 0$,  the iterated cup product maps
$\eta^{d-q}:~H^q(X,\rat) \to H^{2d-q}(X,\rat)$ are isomorphisms. Deduce the primitive Lefschetz decomposition:
for every $q\leq d$, set $H^{q}_{\text{prim}}:= \ke\{\eta^{d-q+1}: H^{q} \to H^{2d-q+2} \}$;  then we have, for every $0 \leq q \leq d$
$H^q = \oplus_{j\geq 0} H^{q-2j}_{\text{prim}}$, and, for $d \leq q\leq 2d$, we have  $H^q= \eta^{q-d} \cup (\oplus_{j\geq 0} H^{q-2d-2j}_{\text{prim}}).$
Let $f : X\to Y$ be as in (\ref{eq1}), i.e. smooth and projective and let $d:= \dim{X}-\dim{Y}.$ Apply the hard Lefschetz theorem to the fibers
of the smooth map $f$ and deduce the analogue of the primitive Lefschetz decomposition for the direct image sheaves
$\sq:=Rf_*\rat_X$. Argue that in order to deduce (\ref{eq1}) it is enough to show the differentials ${d_r}$ of the spectral sequence
vanish on $H^p(Y, \sq_{\text{prim}})$ for every $q \leq d.$ Use the following commutative diagram, with some entries left blank   on purpose
for you to fill-in, to deduce that  indeed we have that vanishing:
\[\xymatrix{
H^?(Y, \sq_{\text{prim}}) \ar[r]^{d_?}  \ar[d]^{\eta^?} & H^? (Y, ?) \ar[d]^{\eta^?} \\
H^?(Y, ?) \ar[r]^{d_?}   & H^? (Y, ?).
}
\]
(Hint: the right  power of $\eta$ kills a primitive in degree $q$, but is injective in degree $q-1$.)
Remark: the refined decomposition (\ref{eq1a}) is proved in a similar way by replacing the spectral sequence
above with the analogous  one for ${\rm Hom}(\sq[-q], Rf_* \rat_X)$: first you prove it is $E_2$-degenerate; then
you lift the identity $\sq \to \sq$ to a map in $\mbox{Hom}(\sq [-q], Rf_* \rat_X)$ inducing the identity on $\sq$;
see \ci{shockwave}.
}
\end{exe}

{\bf Heuristics for $E_2$-degeneration and for semisimplicity of the $R^qf_* \rat_X$  via weights.} (What follows should be taken with a big grain of salt.)
 It seems that Deligne guessed at $E_2$-degeneration  by looking at the same situation over the algebraic closure of a finite field by considerations (``the yoga of weights'' \cites{ho1,poids})  of  the size (weight) of the eigenvalues of   action of 
 Frobenius  on the entries $E_r^{pq}$: they should have weight something analogous to $\exp{(p+q)}$ (we are using the exponential function as an analogy only, one needs to say more, but we shan't) so the Frobenius-compatible differentials must be zero.
There is a similar heuristics for the Deligne's theorem to the effect that the $\sq$ are semisimple:
if $0 \to M \to \sq \to N \to 0$ is a short exact sequence,  then Frobenius acts  on  $Ext^1(N,M)$ with weight $\exp{(1)}$; take 
$M\subseteq \sq$ to be the maximal semisimple subobject; then the corresponding extension  is invariant under Frobenius and has weight zero $(1=\exp{(0)})$;
it follows that the extension splits and  
the biggest semisimple  in $\sq$ splits off: $\sq \cong M \oplus N$; if the resulting quotient $N$ were non trivial,  then it would contain a non trivial simple that then, by the splitting, would enlarge
the biggest semisimple $M$ in $\sq$; contradiction. This kind of heuristics is now firmly based in deep theorems
by Deligne and others \cites{weil2, bbd} for varieties finite fields and their algebraic closure, and by M. Saito
\ci{samhm} in the context of mixed Hodge modules
over complex algebraic varieties.

\begin{exe}\la{loc1}{\rm 
({\bf  Rank one locally constant sheaves}) 
Take $[0,1] \times \rat$ and identify the two ends by multiplication by $-1$. Interpret this as a 
rank one  locally constant sheaf on $S^1$ that is not constant. Do the same, but multiply by $2$.
Do the same, but first replace $\rat$ with $\overline{\rat}$ and multiply by a root of unity.
Show that the tensor product operation $(L,M) \to L\otimes M$ induces the structure of an abelian
group on the set of isomorphisms classes of rank one locally constant sheaves on a  variety $Y$. Determine the torsion elements of this group when you replace $\rat$ with
$\overline{\rat}$. Show that  if we replace $\rat$ with $\comp$ (this is the
character variety for rank one complex representations) we obtain the structure of a  complex Lie group.
}
\end{exe}

\begin{exe}\la{hy6}
{\rm 
({\bf Locally constant sheaves and representations of the fundamental group})
 A  locally constant sheaf (a.k.a. local system)  $L$ on $Y$ gives rise to a representation $\rho_L: \pi_1 (Y, y)
\to GL(L_y)$: pick a loop $\gamma (t)$ at $y$ and use   local trivializations of  $L$ along the loop to  move
vectors in $L_y$ along  $L_{\gamma (t)}$, back to $L_y$. 
}
\end{exe}

\begin{exe}\la{ow}{\rm
({\bf Representations of the fundamental group  and locally constant sheaves})  Given a representation $\rho: \pi_1 (Y, y) \to GL (V)$ into
a finite dimensional vector space, consider a universal cover $(\w{Y}, \w{y}) \to (Y,y),$
build  the quotient space $(V  \times \w{Y})/\pi_1(Y,y)$, take the natural map (projection)  to $Y$
and take the sheaf of its local sections. Show that this is a locally constant sheaf whose associated 
representation is $\rho.$
}
\end{exe}

\begin{exe}{\rm({\bf Zeroth cohomology of a local system})
Let $X$ be a connected space. Let $L$ be a local system on $X$, and write for $M$ for the associated $\pi_1(X)$ representation. Show that
\[ H^0(X; L) = M^{\pi_1(X)}, \]
where the right hand side is the fixed part of $M$ under the $\pi_1(X)$-action.
}
\end{exe}

\begin{exe}{\rm ({\bf Cohomology of local systems on a circle})
Fix an orientation of $S^1$ and the   generator $T\in \pi_1(S^1)$ that comes with it. Let $L$ be a local system on $S^1$ with associated monodromy representation $M$. Show that
\[ H^0(S^1, L) = \mathrm{ker}((T-id)\colon M \to M), \quad 
H^1(S^1, L) = \mathrm{coker}((T-id)\colon M\to M), \]
and that $H^{>1}(S^1, L) = 0$. 

\n
(Hint: one way to proceed is to use Cech cohomology. Alternatively, embed $S^1$ as the boundary of a disk and use relative cohomology (or dualize and use compactly supported cohomology, where the orientation is easier to get a handle on).
}
\end{exe}

\begin{exe} {\rm ({\bf Fiber bundles over a circle: the Wang sequence})
This is an extension of the previous exercise. Let $f\colon E\to S^1$ be a locally trivial fibration with fibre $F$ and monodromy isomorphism $T\colon F\to F$. Show that the Leray spectral sequence gives rise to a short exact sequence
\[ 0 \to H^1(S^1, R^{q-1}f_*\mathbb{Q}) \to H^q(E, \mathbb{Q}) \to H^0(S^1, R^qf_*\mathbb{Q}) \to 0.\]
Use the previous exercise to put these together into a long exact sequence
\[ \ldots \to  H^q(E, \mathbb{Q}) \lorw H^q(F, \mathbb{Q}) \lorw H^q(F, \mathbb{Q}) \to H^{q+1}(E) \to \ldots \]
where the middle map $H^q(F; \mathbb{Q}) \to H^q(F;\mathbb{Q})$ is given by $T^*-id$. Make a connection with the theory of nearby cycles.
}
\end{exe}

\begin{exe}\la{098}{\rm 
({\bf The abelian category $Loc (Y)$})
Show that  the abelian category 
$Loc\,(Y)$ of locally constant sheaves of finite rank  on $Y$ is equivalent to the abelian category of finite dimensional $\pi_1(Y,y)$-representations. Show that both categories are noetherian (acc ok!), artinian (dcc ok!) and have a duality anti-self-equivalence.
}
\end{exe}

Exercise~\ref{vbt} below   is in striking contrast with the category $Loc$, but also with the one  of perverse sheaves, which admits, by its very definition, the anti-self-equivalence given by Verdier duality.

\begin{exe}\la{vbt}{\rm 
({\bf The abelian category $Sh_c (Y)$ is not artinian.})
Show that in the presence of such an anti-self-equivalence, noetherian is equivalent to artinian. Observe that the category
$Sh_c(Y)$ whose objects are the constructible sheaves (i.e. there is a finite partition of $Y=\coprod Y_i$ into locally closed subvarieties
to which the sheaf restricts to a locally constant one (always assumed to be  of finite rank!) is  abelian and noetherian, but  it is not artinian. Deduce that $Sh_c(Y)$
does not admit an anti-self-equivalence.  Give an explicit example of the failure of dcc in $Sh_c(Y)$.
Prove that $Sh_c(Y)$ is artinian IFF $\dim Y =0.$
}
\end{exe}

\begin{exe}{\rm
({\bf Cyclic coverings}) Show that the  direct image sheaf
sheaf $R^0f_*\rat$ for the map $S^1 \to S^1$, $t \to t^n$ is a  semisimple locally constant sheaf of rank $n$;
find its simple summands (one of them is the constant sheaf $\rat_{S^1}$  and the resulting splitting
is given by the trace map). Do the same for locally constant sheaves with $\overline{\rat}$ coefficients.}
\end{exe}

\begin{exe}{\rm
({\bf Indecomposable non simple}) The rank two locally constant sheaf on $S^1$ given by the non trivial unipotent 
$2\times 2$ Jordan block is indecomposable, not simple, not semisimple. Make a connection
between this locally constant sheaf and the Picard-Lefschetz formula for the 
degeneration of a curve of genus one to a nodal curve.
}
\end{exe}

{\bf Amusing monodromy dichotomy.} There is an important and amusing dichotomy concerning local systems in algebraic geometry (which we state informally):
the global local systems arising in complex algebraic  are semisimple (i.e. completely reducible;  related to  Zariski closure the image of the fundamental group
in the general linear group being reductive); the restriction of these local systems to small punctured disks with centers at infinity
(degenerations),  are quasi-unipotent, i.e. unipotent after
taking a finite cyclic covering if necessary.  This local quasi-unipotency is in some sense the opposite of the global complete reducibility.

{\bf Quick review of Deligne's 1972 and 1974 theory of mixed Hodge structures \cites{ho1,ho2,ho3}; see also \ci{durfee}.}
Deligne discovered the existence of a remarkable structure, a mixed Hodge structure,  on the singular cohomology of  a complex algebraic variety $X$:
there is an increasing filtration $W_kH^*(X,\rat)$ and a decreasing filtration $F^p H^*(X,\comp)$  (with conjugate filtration denoted by $\overline{F}$) such that   the graded quotients $Gr^W_k H^*(X,\comp) = \oplus_{p+q=k}H^{pq}_k$, where the splitting 
is induced by the (conjugate and opposite)  filtrations $F, \overline{F}$; i.e. $(W, F, \ov{F})$  induce  pure Hodge structure of weight $k$
on $Gr^W_k$. This structure is  canonical and functorial for maps of complex algebraic varieties.
Important: one has that a map of mixed Hodge structures $f: A\to B$ is automatically {\em strict}, i.e.
if $f(a) \in W_kB,$ then there is $a'\in W_kA$ with $f(a')=f(a)$.   Kernels, images, cokernels
of pull-back maps in cohomology inherit such a structure.
If $X$ is a projective manifold, we get the known Hodge $(p,q)$-decomposition: $H^i(X,\comp) = \oplus_{p+q =i}H^{pq}(X)$. It is important to take note that for each fixed $i$ we have 
  $H^i(X,\comp) = \oplus_k Gr^W_k H^i(X,\comp) = \oplus_k \oplus_{p+q=k} H^{pq}_k(X)$ which may admit
  several non zero $k$ summands for $k\neq i$.  In this case, we say that the mixed Hodge structure  is mixed. This happens
for the projective nodal cubic: $H^1= H_0^{0,0}$, and for the punctured affine line $H^1=H_2^{1,1}$.
Here are some ``inequalities'' for the weight filtration: $Gr^W_kH^d=0$ for $k\notin [0,2d]$; if $X$ is complete, then $Gr^W_{k>d}H^d=0$; 
if $X$ is nonsingular, then $Gr^W_{k<d}H^d=0$ and $W_dH^d$ is the image of the restriction map
from any nonsingular completion (open immersion in proper nonsingular); if $X\to Y$ is surjective and $X$ is complete nonsingular, then the kernel of the pull-back
to $H^d(X)$ is $W_{d-1}H^d(Y).$ 

\begin{exe}\la{rty60} {\rm ({\bf Amazing weights})
Let $Z \to U \to X$ be a closed immersion with  $Z$ complete  followed by an open dense  immersion  into a complete nonsingular variety.
Using some of the weight inequalities listed above, together with strictness, to show that
the images of $H^*(X,\rat)$ and $H^*(U,\rat)$ into $H^*(Z,\rat)$ coincide. 
Build a counterexample in complex geometry (Hopf!). Build a counterexample in real algebraic geometry (circle, bi-punctured sphere, sphere).
}
\end{exe}

The reader is invited to produce an explicit example of a projective normal surface having mixed singular cohomology.
Morally speaking, as soon as you leave the world of projective manifolds and dive into the one of projective varieties, ``mixedness'' is the norm.

For an explicit example of  a proper map with no cohomological decomposition
analogous to (\ref{eq1}),  see Exercise~\ref{bgt5}.
We can produce many by pure-thought
using Deligne's theory of mixed Hodge structures. Here is how.
\begin{exe}{\rm 
({\bf In general, there is no decomposition $Rf_* \rat_X\cong \oplus R^qf_*\rat_X [-q]$})
Pick a normal projective variety $Y$
whose singular cohomology is a non pure  mixed Hodge structure.  Resolve the singularities
$f: X \to Y.$  Use  Zariski main theorem to show that $R^0f_* \rat_X = \rat_Y.$  Show that, in view of the
 the mixed-not-pure assumption, the map of mixed Hodge structures $f^*$
is not injective. Deduce that $\rat_Y$ is not a direct summand of $Rf_*\rat_X$ and that, in particular,
there is no decomposition $Rf_* \rat_X \cong \oplus R^qf_* \rat_X [-q]$ in this case.
(In some sense, the absence of such a  decomposition is the norm for proper maps of varieties.)
}
\end{exe}

\begin{exe}\la{fcv}
{\rm ({\bf The affine cone $Y$ over a projective manifold $V$}) 
Let $V^d\subseteq \bb{P}$ be an embedded projective manifold  of dimension $d$ and let $Y^{d+1} \subseteq
\bb{A}$ be its affine cone with vertex $o.$ Let $j: U:= Y\setminus \{o\} \to Y$ be the open embedding.
Show that $U$ is the $\comp^*$-bundle over $V$ of the dual to the hyperplane line bundle
for the given embedding $V \subseteq \bb{P}$.  Determine $H^*(U,\rat)$. 
Answer: for every for $0 \leq q  \leq d$, $H^{q}(U) = H^q_{Prim}(V)$  and $H^{1+d +q}(U) = H^{d-q}_{Prim}(V).$
Show that  $R^0j_*\rat_U=\rat_Y$  and that, for $q>0$, $R^{q}j_*\rat _U$ is skyscraper at $o$ with stalk
$H^q(U,\rat)$. Compute $H^q_c(Y,\rat)$. Give a necessary and sufficient condition on the cohomology of $V$ that ensures
that $Y$  satisfies Poincar\'e duality $H^q(Y,\rat)\cong
H^{2d+2-q}_c(Y,\rat)$. Observe that if $V$ is a curve this condition boils down to it having genus zero.
Remark: once  you know about a bit about Verdier duality, this exercise tells you that the complex $\rat_Y[\dim{Y}]$ is Verdier self dual
iff $V$ meets the condition you have identified above; in particular, it does not if $V$ is a curve of positive genus.
}
\end{exe}

\begin{fact}\la{fcvo}
{\rm ({\bf  $\m{IC}_Y$, $Y$ a cone over a projective manifold $V$}) Let things be as in Exercise~\ref{fcv}.
By adopting the definition of the intersection complex
as an iterated push-forward followed by truncations, as originally given by Goresky-MacPherson,
the intersection complex of $Y$ is defined to be  $\m{IC}_Y:= \td{d} Rj_* \rat_U$,
where we are truncating the image direct complex $Rj_* \rat_U$  in the following way: keep the same entries up to degree $d-1$, replace the $d$-th entry
by the kernel of the differential exiting it and setting the remaining entries to be zero; the resulting cohomology sheaves
are the same as the ones for $Rj_* \rat_U$ up to degree $d$ included, and they are zero afterwards.
 More precisely, the   cohomology sheaves  of this complex are as follows:
$\m{H}^q=0$ for $q \notin [0, d]$, 
$\m{H}^0 = \rat_Y$,   and for $1 \leq q \leq d$, $\m{H}^i$ is skyscaper at $o$ with stalk $H^q_{Prim}(V,\rat)$.
Here is a justification for this definition: while $\rat [\dim{Y}]$ usually fails to be  Verdier self-dual, one
can  verify directly that the intersection complex $IC_Y:= \m{IC}_Y [\dim{Y}]$ is Verdier self-dual. 
In general, if we were to truncate at any other
spot, then we would not get this self-duality behavior (unless we truncate at minus 1 and   get zero).
Note that the knowledge of the cohomology sheaves of a complex, e.g. $\m{IC}_Y$, is important information, but
it does not characterize the complex up to isomorphism.
}
\end{fact}

\begin{exe}\la{bgt5} {\rm ({\bf Example of no decomposition  $Rf_* \rat_X\cong \oplus R^qf_*\rat_X [-q]$})
Let things be as in Exercise~\ref{fcv} and assume that $V$ is a curve of positive genus, so that $Y$ is a
surface.
Let $f: X \to Y$ be the  resolution  obtained by blowing up the vertex $o \in Y.$
Use the failure of the self-duality of $\rat_Y[2]$ to deduce that $\rat_Y$ is not a direct summand
of $Rf_*\rat_{X}$. Deduce that $Rf_*\rat_X \not\simeq \oplus R^qf_*\rat_X [-q]$.
(The reader is invited to check  out arXiv:math/0504554, \S3.1: it is an explicit computation dealing with this example showing that
as you try split $\rat_Y$ off $Rf_* \rat_{X}$, you meet an obstruction;  instead, you end up
splitting $\m{IC}_Y$ off $Rf_* \rat_{X}$, provided you have defined $\m{IC}_Y$   as the truncated push-forward as above.)
}
\end{exe}

\begin{fact}\la{509}
{\rm
({\bf  Intersection complexes on curves})
Let $Y^o$ be a nonsingular curve and $L$ be a locally constant sheaf on it. Let $j:Y^o\to Y$ be an open immersion
into another curve (e.g. a compactification). Then $IC_Y (L) = j_*L[1]$ (definition of $IC$ via push-forward/truncation). Note that
if $y \in Y$ is a  nonsingular point, then the stalk  $(j_*L)_y$ is given by the local monodromy invariants
of $L$ around a small loop about $y.$  The complex $Rj_*L[1]$ may fail to be Verdier self-dual, whereas
its truncation $\td{-1} Rj_* L[1]  = j_* L[1] = IC_Y(L)$ is Verdier self-dual.
Note that we have a factorization $Rj_! L[1] \to IC_Y(L) \to Rj_* L[1].$ This is not an ``accident'': see  the end of \S\ref{intco1}.
}
\end{fact}

\begin{exe}\la{blowups}
{\rm  ({\bf Blow-ups}) Compute the direct image sheaves $R^qf_* \rat$ for the blowing-up
of $\comp^m \subseteq \comp^n$(start with $m=0$; observe that there is a product
decomposition of the situation that allows you to reduce to the case $m=0$). Same question for the composition of the  blow up of $\comp^1 \subseteq \comp^3$,
followed by the blowing up of a positive dimensional fiber of the first blow up. Observe  that
in all cases,
one gets an the decomposition $Rf_* \rat \cong \oplus R^qf_*\rat[-q].$
 Guess  the shape of the decomposition theorem in both cases.
}
\end{exe}

\begin{exe}\la{g0g0}{\rm
({\bf Examples of the decomposition theorem})
Guess  the exact form of the
cohomological  and ``derived'' decomposition theorem in the following cases:
1) the normalization of a cubic curve  with a node and of a cubic curve with a cusp;
 2) the blowing up of a smooth
subvariety of an algebraic manifold; 3) compositions of various iterations of blowing ups of nonsingular varieties along   smooth centers; 
4) a  projection $F \times Y \to Y$; 5) the blowing up of the vertex
of the affine cone over the nonsingular quadric in $\pn{3}$; 6) same but  for the 
projective cone; 7) blow up the same affine and projective cones but along 
a plane through the vertex of the cone; 8) the blowing up of the vertex of the affine/projective  cone over  an embedded projective manifold.}
\end{exe}

\begin{exe}\la{lefpe} {\rm ({\bf DT for Lefschetz pencils})
Guess  the shape of the dt for a Lefschetz pencil  $f: \w{X} \to \pn{1}$ on a nonsingular projective surface $X$.
Word out explicitly   the invariant cycle theorems in this case. Do the same for
a nonsingular projective manifold. When do we get skyscraper contributions?
}
\end{exe}

\begin{exe}\la{kunneth} {\rm ({\bf K\"unneth for the derived image complex})
One needs a little bit of working experience with the derived category to carry out what below. But try anyway.
Let $f: X:=Y\times F \to Y$.
A cohomology class $a_q \in H^q(X,\rat)$ is the same thing as a map in the derived category
$a_q:\rat_X \to \rat_X[q]$.  By pushing forward via $Rf_*$, by observing that $Rf_* f^* \rat_Y = Rf_*\rat_X$,
 by pre-composing with the adjunction map $\rat_Y \to Rf_* \rat_X$, we get a map $a_q: \rat_Y \to Rf_* \rat_X[q].$
 Take $a_q$ to be of the form $pr_F^*\alpha_q.$ Obtain a map  $\alpha_q \underline{H}^q(F) \to Rf_*\rat_X [q]$.
 Shift to get $\alpha_q:  \underline{H}^q(F)[-q] \to Rf_*\rat_X.$
 Show that the map induces the ``identity'' on the $q$-th direct image sheaf and zero on the other direct image sheaves.
 Deduce that $\sum_q \alpha_q:   \oplus_q \underline{H}^q(F)[-q] \to Rf_*\rat_X$ is an isomorphism in the derived category
 inducing the ``identity'' on the cohomology sheaves. Observe that you did not make any choice in what above.
}
\end{exe}

\begin{exe}{\rm ({\bf Deligne theorem as a special case of the decomposition theorem}) Keeping in mind that if 
$S^o=S$,  then $\m{IC}_S(L) = L,$ recover the Deligne  theorem  from the decomposition theorem.}
\end{exe}

\section{Lecture 2: The category of perverse sheaves \texorpdfstring{$P(Y)$}{P(Y)}}\la{oiu}

{\bf Summary of Lecture 2.}  {\em The constructible derived category. Definition of perverse sheaves. Artin vanishing and its relation to a proof of the Lefschetz hyperplane theorem for perverse sheaves. The perverse t-structure (really, only the perverse cohomology functors!). Beilinson's and Nori's equivalence theorems. Several  equivalent definitions of intersection complexes.}

\subsection{Three Why? And a brief history of perverse sheaves}\la{0i4}

{\bf Why intersection cohomology?}
Let us look at  (\ref{3edd}) for $X$ and $Y$  nonsingular:
\[H^*(X,\rat) \cong  \oplus_{q,\pev_q} IH^{*-q}(S,L),\] i.e. the  l.h.s. is ordinary cohomology,
but the r.h.s. is not any kind of ordinary cohomology on $Y$: we need intersection cohomology to state the decomposition
theorem,
even when $X$ and $Y$ are nonsingular. 
The intersection cohomology groups of a projective variety enjoy  a  battery of wonderful properties
(Poincar\'e-Hodge-Lefschetz package). 
In some sense, intersection cohomology nicely replaces singular cohomology on singular varieties, but with a funny twist: 
singular cohomology is functorial, but has no Poincar\'e duality; intersection cohomology
has Poincar\'e duality, but is not functorial!

{\bf Why the constructible derived category?} 
The cohomological Deligne theorem  (\ref{eq1}) for smooth projective maps is a purely cohomological statement
and it can be  proved via purely cohomological methods (hard Lefschetz + Leray spectral sequence).
The cohomological decomposition theorem  (\ref{3edd}) is also a cohomological statement. However, there is no known proof
of this statement that does not make use of the formalism of the  middle perversity t-structure present in  the constructible
derived category:
one  proves the derived version (\ref{eq2})  and then  deduces the cohomological one (\ref{3edd}) by taking cohomology.
Actually, the definition of perverse sheaves does not make sense if we take the whole derived category, we need to take complexes with
cohomology sheaves supported at closed subvarieties (not just classically closed subsets). We thus restrict to an agreeable, yet
flexible, class of complexes: the ``constructible complexes''.

{\bf Why perverse sheaves?} 
Intersection complexes, i.e. the objects appearing
on both sides of the decomposition theorem  (\ref{eq2}) are very special perverse sheaves. In fact, in a precise way, they form the building blocks
of the category of perverse sheaves: every perverse sheaf is an iterated extension of a collection of intersection complexes.
Perverse sheaves satisfy their own set of beautiful properties: Artin vanishing theorem, Lefschetz hyperplane theorem,
stability via duality, stability via vanishing and nearby cycle functors. 
As mentioned above, the known proofs of the decomposition theorem use the machinery of perverse sheaves.

{\bf A brief history of perverse  sheaves.} 

Intersection complexes were invented by Goresky-MacPherson as a tool to systematize, strengthen and widen the scope of  their own intersection cohomology theory. For example, their original geometric proof of Poincar\'e duality can be replaced by   the  
self-duality property of the intersection complex. See also S. Kleiman's very entertaining  survey \ci{kle}.

The conditions leading to the definition of perverse sheaves appeared first  in connection with  the Riemann-Hilbert correspondence
established  by Kashiwara and by  Mekbouth: their result is an equivalence of categories
between the constructible derived category (which we have been procrastinating to define) and the derived category of regular holonomic D-modules (which
we shall not define);
the standard t-structure, given by the standard truncations met in Exercise~\ref{fcvo}, of these two categories do not correspond to each other under the Riemann-Hilbert equivalence; the conditions leading to the 
``conditions of support'' defining  of perverse sheaves 
are the (non trivial) translation in the constructible derived category  of the conditions  on the D-module side
stating that a complex of D-modules has trivial cohomology D-modules in positive degree. It is a seemingly unrelated, yet 
remarkable  and beautiful fact, that the conditions of support so-obtained are precisely what makes the Artin vanishing Theorem~\ref{artinio}  work 
on an affine variety.

As mentioned above, Gelfand-MacPherson conjectured the decomposition theorem for $Rf_* \m{IC}_X.$ 
Meanwhile Deligne had developed a theory of pure complexes for varieties defined over finite fields and established
the invariance of purity under push-forward by proper maps. 
Gabber proved that the intersection complex
of a pure local system in that context, is pure.  The four authors of \ci{bbd} introduced and developed systematically the
basis for the theory of t-structures, especially with respect to the  middle   perversity. They then proved that
the notions of purity  and perverse t-structure are compatible: a pure complex splits over the algebraic closure of the
finite field as prescribed by the r.h.s. of (\ref{eq2}).  The decomposition theorem over the algebraic closure of a finite field
follows when considering the purity result for the proper direct image mentioned above. The whole Ch. 6 in \ci{bbd},
aptly named ``De $\bb{F}$ \`a  $\comp$'',
is devoted to explaining how these kind of results over the algebraic closure of a finite field
yield results over the field of complex numbers. This established the original proof of the decomposition
theorem
over the complex numbers for semisimple complexes of geometric origin (see \ci{bbd}, 6.2.4, 6.2.5), such as
$\m{IC}_X.$

M. Saito has developed in \ci{samhm} the 
theory of mixed Hodge modules which yields the desired decomposition theorem when $M$ underlies a variation of polarizable pure Hodge
structures.

M.A. de Cataldo and L. Migliorini have given a proof based on classical Hodge theory of the decomposition theorem when
$M$ is constant \ci{decmightam}.

Finally, the decomposition theorem stated in  (\ref{eq2}) is the most general statement currently available over the complex
numbers and is due to  work of C. Sabbah \ci{sabbah} and T. Mochizuki \ci{mochizuki} (where this is done in the essential
 case of projective maps of quasi projective manifolds; it is possible to extend it  to proper maps of algebraic
 varieties).  The methods (tame harmonic bundles,
D-modules)
 are quite different from the ones discussed in these lectures.

\subsection{The constructible derived category \texorpdfstring{$D(Y)$}{D(Y)}}\la{cdcat}

The  decomposition theorem isomorphisms (\ref{eq2}) take place in  the ``constructible derived category'' $D(Y)$. It is probably a good time to try and give an idea
what this category is.

{\bf Constructible sheaf.} A sheaf $F$ on $Y$ is {\em constructible} if there is
 a finite disjoint union decomposition $Y=
\coprod_a S_a$ into locally closed subvarieties 
such that the restriction  $F_{|S_a}$ are locally constant  sheaves of finite rank.  This is a good time to look at Exercise~\ref{cantor}.

{\bf Constructible complex.} A complex $C$  of sheaves   of rational vector spaces on $Y$  is said to be   {\em constructible}
if it is bounded (all but finitely many
of its cohomology sheaves are zero) and  its cohomology sheaves are constructible sheaves. See the most-important
Fact~\ref{moimp}.


 
 {\bf Constructible derived category.}
The definition of $D(Y)$ is kind of a mouthful: {\em it is the full subcategory of the derived category  $D(Sh(Y,\rat))$ of the category of sheaves of rational vector spaces
whose objects are the constructible complexes.}

It usually  takes time to absorb these notions and to absorb the apparatus
it gives rise to. We take a different approach and we try to isolate some
of the aspects of the theory that are more relevant to the decomposition theorem. We do
not dwell on  technical details.

{\bf  Cohomology.}
Of course, the first functors to consider are cohomology and cohomology with compact supports $H^i(Y,-), H^i_c(Y,-): D(Y) \to D(point)$. They can be seen as special cases of derived direct images.

{\bf Derived direct images $Rf_*, Rf_!$.}
The derived direct image $Rf_*, Rf_!: D(X) \to D(Y)$,   for every map $f:X \to Y$.
The first thing to know is that $H^*(X,C)= H^*(Y,Rf_*C)$ and that $H^*_c(X,C)= H^*_c(Y,Rf_! C)$, so that we may view them as generalizing cohomology. 

{\bf Pull-backs.} The pull-back  functor $f^*$ is  probably the most intuitive one. The extraordinary pull-back  functor $f^!$
is tricky and  we will not dwell on it. It is the right adjoint to $Rf_!$; for open immersions, $f^!=f^*$; for closed immersions
it is the derived versions of the sheaf of sections supported on the closed subvariety; for smooth maps of relative dimension
$d$, $f^!=f^*[2d]$. A down-to-earth reference for $f^!$ and duality  I like is \ci{iv} (good also, among other things,
as an introduction to Borel-Moore homology). I also like \ci{gel-man}. There is also the seemingly inescapable,
and nearly encyclopedic \ci{k-s}.

\begin{fact}\la{moimp}{\rm A good reference is \ci{borel}.
Given $C \in D(Y)$, and $y\in Y$ there is a system of ``standard neighborhoods'' $U_y(\e)$ (think of $0 < \e \ll 1$ as the radius of an Euclidean ball; of course, our $U_y(\e)$ are singular, if $Y$ is singular at $y$) such that $H^*(U_y(\e), C)$ and $H_c^*(U_y(\e), C)$ are ``constant'' (make the meaning of constant precise) in $ 0 < \e \ll 1$. The $U_y(\e)$ are cofinal
in the system of neighborhoods of $y.$ We have a canonical identification $\m{H}^*(C)_y = H^*(U_y(\e), C)$ for $0 < \e \ll 1$.
This is very important for many reasons. Let us give one. Let $f:X\to Y$ be a map and let $K\in D(X)$. Since $Rf_* K$ is constructible,  we have $(R^*f_* K)_y= H^*(f^{-1}(U_y(\e), K)$ for $0 < \e \ll 1$. Caution: there is a natural map
to $H^*(f^{-1}(y), K)$, but this map is in general neither surjective, nor injective; if the map is proper, then it is
an isomorphism (proper base change). 
}
\end{fact}

{\bf Verdier duality.} This is an anti-self equivalence $(-)^\vee :D(Y)^{op} \cong D(Y)$
whose defining property is a natural perfect pairing 
$H^*(Y, C^\vee) \times H^{-*}_c (Y, C) \lorw \rat$, or, equivalently, of a canonical isomorphism
\beq\la{devd} H^*(Y, C^\vee) \cong  H^{-*}_c (Y, C)^{\vee}.\eeq
If $Y$ is nonsingular irreducible, then $\rat_Y^\vee= \rat_Y [2\dim{Y}]$, 
and we get an identification   $H^{*+2\dim Y}(Y, \rat)=  
H^{-*}_c(Y, \rat)^\vee$, i.e. Poincar\'e duality.

{\bf Stability of constructibility.}
It is by no means obvious, nor easy, that if $C$ is constructible, then  $Rf_*C$ is constructible.
This can be deduced from the Thom isotopy lemmas\footnote{The main two points are: 1)
given a map $f: X\to Y$ of complex algebraic varieties, there is a disjoint union decomposition $Y=\coprod_i Y_i$
into locally closed subvarieties such that $X_i:= f^{-1}Y_i \to Y_i$ is a topological fiber bundle for the classical topology;
2) algebraic maps can be completed  compatibly with the previous assertion (you may want to sit down and come up with a reasonable precise statement yourself)} \ci{gomacsmt}. 
This is a manifestation   of the important principle
that constructibility is preserved under all the ``usual'' operations on the derived category 
of sheaves on $Y$ (see \ci{borel}). The list above is more complete once we include the derived $\m{RH}om$, the tensor product
of complexes (it is automatically derived when using $\rat$-coefficients), the nearby and vanishing cycle functors, etc.

 {\bf Duality exchanges.} Verdier duality exchanges $Rf_*$ with $Rf_!$, and $f^*$ with $f^!$,
i.e.,   $Rf_* (C^\vee) = (Rf_! C)^\vee$ and $f^* (K^\vee) = (f^! K)^\vee$.
Here is a nice consequence: let $f$ be proper (so that $Rf_* = Rf_!)$, then if $C$ is self-dual, then so is $Rf_* C.$

{\bf The importance of being proper.} Proper maps are important for many reasons:  $Rf_!=Rf_*$;  the duality exchanges simplify;
the proper base change theorem holds, a special case of which tells us that $(R^qf_*\rat_X)_y = H^q(f^{-1}(y),\rat)$ (see Exercise~\ref{loccst});
$Rf_*$ preserves pure complexes (Frobenius, mixed Hodge modules);
Grothendieck trace formula is about $Rf_!$; the decomposition theorem is about proper maps.

{\bf Adjunctions.}
We have adjoint pairs
$(f^*,Rf_*)$ and $(Rf_!,f^!)$, hence natural transformations: $Id \to Rf_* f^*$
$Rf_!f^! \to Id.$ By applying cohomology to the first one, we get the pull-back map in cohomology,
and by applying cohomology with compact supports to the second, we obtain the push-forward 
in cohomology with compact supports. $\m{RH}om$ and $\otimes$ also form an adjoint pair (you should
formalize this).

{\bf The attaching triangles and the long exact sequences we already know.} By combining adjunction maps, we get some familiar situations from algebraic topology.
 Let $j: U \to Y \leftarrow Z: i$
be a complementary pair of open/closed embeddings. Given $C \in D(Y)$, we have
the distinguished triangle $i_*i^! C \to C \to j_*j^*C \to i_!i^! C[1]$ and, by applying
sheaf cohomology, we get the long exact sequence of relative cohomology
\[
\ldots \to H^q(Y,U,C) \to H^q(Y,C) \to H^q(U,C_{|U}) \to H^{q+1} (Y,U,C) \to \ldots
\]
If $g: Y \to Z$ is a map, then we can push forward  the attaching (distinguished) triangle
and obtain a distinguished triangle, which will also give rise to hosts of long exact sequences
when fed to cohomological functors. By dualizing, we get $j_!j^* C \to C \to i_*i^* C \to j_!j^* C[1]$
and by taking cohomology with compact supports, we get the long exact sequence of cohomology with compact supports
\[\ldots \to  H^q_c(U,C) \to H^q_c(Y,C) \to H^q_c(Z,C_{|Z}) \to H^{q+1}_c (U,C) \to \ldots
\]  This is nice, and used very often,
because it gives the usual  nice relation between the compactly supported Betti numbers of the three varieties
$(U,Y,Z).$

Exercise~\ref{conecve} asks you to use the first attaching triangle and its push-forward when studying
the resolution of singularities  of a germ of an isolated surface singularity. This  is an important example: it shows
how the intersection complex of the singular surface arises; it relates the non degeneracy of the intersection form
on the curves contracted  by the resolution, to the decomposition theorem for the resolution map. 
The careful study of this example  allows to extract many general ideas and patterns, specifically how to relate the non degeneracy of certain local intersection forms to a proof of the decomposition theorem.

\subsection{Definition of perverse sheaves}\la{defps}
{\bf Before we define perverse sheaves.}
The category of perverse sheaves on an algebraic variety is abelian, noetherian, anti-self-equivalent under Verdier duality, artinian.
The cohomology groups  of perverse sheaves
satisfy Poincar\'e duality, Artin vanishing theorem and  the Lefschetz hyperplane theorem. They are stable under the nearby and
vanishing cycle functors.
The simple perverse sheaves, i.e. the intersection complexes with simple coefficients,
satisfy the decomposition theorem and the relative hard Lefschetz theorem, and their cohomology groups satisfy
the Hard Lefschetz theorem and, when the coefficients are ``Hodge-theoretic'', the Hodge-Riemann bilinear relations.
Perverse sheaves play important roles in the topology of algebraic varieties, arithmetic algebraic geometry,
singularity theory, combinatorics, representation theory,
geometric Langlands program.

Even though not everyone agrees (e.g. I), one may say that perverse sheaves are more natural than constructible sheaves; see Remark~\ref{b6y}.
Perverse sheaves  on singular varieties are close in spirit to locally constant sheaves on algebraic manifolds.

Perverse sheaves  have  at least one drawback: they are not sheaves!\footnote{The ``sheaves'' in perverse sheaves is because the objects and the arrows in $P(Y)$
can be glued from local data, exactly like sheaves; this is false for the derived category and for $D(Y)$!;  do you know why?
As to the term ``perverse'', see M. Goresky's post on Math Overflow: What is the etymology of the term perverse sheaf?}

{\bf The conditions of support and of co-support.}
We say that a constructible complex
$C \in D(Y)$ satisfies the {\em conditions of support} if $\dim \, supp \,\m{H}^i (C)
\leq -i$, for every $i \in \zed$, and  that it satisfies the {\em conditions of co-support} if its Verdier dual $C^\vee$ satisfies
the conditions of support.

\begin{defi}\la{deps} {\rm ({\bf Definition of the category $P(Y)$ of perverse sheaves})
We say that $P\in D(Y)$ is a perverse sheaf if  $P$ satisfies the conditions of support and of co-support.
The category $P(Y)$ of perverse sheaves on $Y$
is the full subcategory of the constructible derived category $D(Y)$ with objects the  perverse sheaves.
}
\end{defi}

{\bf Conditions of (co)support  and vectors of dimensions.}
To fix ideas, let $\dim{Y}=4$. The following  vector exemplifies
the upper bounds for the dimensions
of the supports of the cohomology sheaves $\m{H}^i$ for $-4\leq i \leq 0$ (outside of this interval, the cohomology sheaves 
of a perverse sheaf can be  shown to be zero)
$(4,3,2,1,0)$. For comparison, the analogous vector  for an intersection complex of the form $IC_Y(L)$ is $(4,2,1,0,0)$.
For a table giving  a good visual for the conditions of support and co-support  for perverse sheaves and for intersection complexes,
see \ci{bams}, p.556.

It is important to keep in mind that the support conditions are conditions on the stalks
(direct limits of cohomology over neighborhoods)
of the cohomology sheaves, whereas the conditions of co-support are conditions on the co-stalk
(inverse limits of cohomology with compact supports over neighborhoods)
and as such are maybe  a bit  less intuitive: see Exercise~\ref{fcv}, where this issue is
tackled for the constant sheaf on the   affine cone over an embedded  projective manifold.  
On the other hand, if for some reason we know that  complex $C \in D(Y)$  is Verdier self-dual, then it is perverse iff it satisfies the conditions of support.
Note that the derived direct image via a  proper map of a self-dual complex is self-dual. This  simple remark  is 
very helpful in practice.

This is a good time to carry out  Exercises~\ref{0101} and~\ref{bt5}. 

\subsection{Artin vanishing and Lefschetz hyperplane theorems}\la{avlht}

{\bf Conditions of support and Artin vanishing theorem.}
The conditions of support seem to have first appeared in the proof of the Artin vanishing theorem for constructible sheaves
in SGA 4.3.XIV, Th\'eor\`eme 3.1, p.159:
{\em let  $Y$ be  affine and $F \in Sh_c(Y)$;   then
$H^*(Y, F) =0$ for $* > \dim Y.$} Note that $F[\dim{Y}]$ satisfies the conditions
of support, and that Artin's wonderful proof works for a constructible complex satisfying the conditions of support.
Note also  that if $Y$ is nonsingular and $F=\rat_Y$,
the result is affordable by means  of Morse theory\footnote{The singular case is one of the reasons
for the existence of the book \ci{gomacsmt}}!  

Since perverse sheaves satisfy the conditions of support and co-support, we see that they automatically
satisfy the ``improved'' version of the Artin vanishing  theorem.

\begin{tm}\la{artinio}  {\rm ({\bf Artin vanishing theorem for perverse sheaves})}\footnote{Note that there is no sheaf analogue of Artin vanishing
for compactly supported cohomology: the Verdier dual of a constructible sheaf is not a sheaf!} Let $P$ be a perverse sheaf on the affine variety $Y.$
Then $H^* (Y, P) = 0,$ for $* \neq [-\dim{Y},0]$ and 
$H^*_c (Y, P) = 0, $ for $* \neq [0,\dim{Y}]$.
\end{tm}
{\em ``Proof.''} See Exercise~\ref{avtpf}.

Exactly as in the Morse-theory proof of the weak Lefschetz theorem, we have that
the improved Artin vanishing theorem implies 

\begin{tm}\la{lht}{\rm (\textbf{Lefschetz hyperplane theorem for perverse sheaves}).} 
Let $Y$ be  a quasi projective variety, let  $P \in P(Y)$ and let  $Y_1\subseteq Y$ a general hyperplane section.
Then $H^*(Y, P) \to H^*(Y_1, P_{|Y_1})$ is an  isomorphism  for $* \leq -2,$ and injective for $*=-1$.
There is a similar statement for compactly supported cohomology (guess it!).
\end{tm}
\begin{proof}  We give the proof in the case when $Y$ is projective. Let $j:U:=Y\setminus Y_1 \to Y \leftarrow Y_1:i$ be the natural maps. We have the attaching triangle
$Rj_! j^*P \to P \to Ri_* i^* P \to$ and the long exact sequence of cohomology (= cohomology with compact supports because $Y$ is projective!) 
\[\ldots \to H^{-k}_c(Y, Rj_!j^*P) \to H^{-k} (Y, P) \to H^{-k} (Y_1, P_{|Y_1}) \to H^{-k+1} (Y, Rj_!j^*P) \to \ldots \]
We need to show that $H^{-k}_c (Y, Rj_!j^*P)= H^{-k}_c (U, j^*P)=0$ for $-k<0$, but this is Artin vanishing
for perverse sheaves on the affine $U$. 
\end{proof}

\begin{rmk}\la{b6y} {\rm  (\textbf{LHT: Perverse sheaves vs. sheaves})
As the proof given above shows,
once we assume the projectivity of $Y$, any hyperplane section $Y_1$  will do. This is similar
to the classical proof of the  Lefschetz hyperplane theorem due to Andreotti-Frankel  (following a suggestion by Thom) and contained in Milnor's Morse Theory (jewel) book, where,   if $Y$ is projective, then we only need to pick a hyperplane section
that contains all the singularities of $Y$, so that the desired vanishing stems from Lefschetz duality and from Morse theory.
This shows that even the constant sheaf  is not well-behaved on singular spaces! If we try and repeat the proof
above for the constant sheaf on a singular space, we stumble into the realization that we do not have the necessary Artin vanishing for cohomology
with compact supports  for the constant sheaf
on the possibly singular  $U$. This issue disappears if we use perverse sheaves!
}
\end{rmk}

\textbf{The simple perverse sheaves are the intersection complexes.} Since the category $P(Y)$ is artinian, every perverse sheaf $P \in P(Y)$
admits an increasing finite filtration with quotients simple perverse sheaves. The important fact is that
the simple perverse sheaves are the intersection complexes $\m{IC}_S(L)$ seen above with $L$ simple! (Exercise~\ref{artinian}).
There is a shift involved: $\m{IC}_S(L)$ is not perverse on the nose, but if  we set
\beq\la{theshift}IC_S(L):= \m{IC}_S(L) [\dim{S}],
\eeq
then the result is  a perverse sheaf  on $Y$.
One views these
complexes on the closed subvarieties $i:S\to Y$ as complexes on $Y$ supported on $S$ via $i_*$ (we do not do this to simplify the notation).

There is a battery of results in intersection cohomology that generalize  to the singular setting
the beloved classical results that hold for complex algebraic manifolds. See Exercise~\ref{tuch}.

\subsection{The perverse t-structure}\la{r6000}
The constructible derived category $D(Y)$ comes equipped with the standard t-structure, i.e. the truncation functors are the standard ones,
and  whose heart is the abelian category
$Sh_c(Y) \subseteq D(Y)$ of constructible sheaves. A t-structure on a triangulated category  is an abstraction of the notion of standard
truncation \ci{bbd}. A triangulated category may carry several inequivalent  t-structures.

\textbf{The middle perversity t-structure on $D(Y)$.}
The category of perverse sheaves $P(Y)$ is also the heart of a t-structure on $D(Y)$, the middle-perversity t-structure.
Instead of dwelling on the axioms, here is a short discussion. 

\textbf{The perverse sheaf cohomology functors.} Every  t-structure on a triangulated category comes with its own
cohomology functors; the standard one comes with the cohomology sheaves functors.
 The perverse t-structure then comes with the perverse cohomology sheaves
 $\pcs^{i}: D(Y) \to P(Y)$ which are  of course $\ldots$ cohomological, i.e. 
turn distinguished triangles into long exact sequences
\[
A \to B \to C \to A[1]  \Longrightarrow \ldots \to \pc{i}{A}\to \pc{i}{B} \to \pc{i}{C} \to \pc{i+1}{A} \to \ldots,
\]
and, moreover, we have
\beq\la{efbt}
\pc{i}{C[j]}= \pc{i+j}{C}.
\eeq
Let us mention that $C\in D(Y)$ satisfies the conditions of support iff its perverse cohomology sheaves
are zero in positive degrees; similarly, for the conditions of co-support (swap positive with negative).

\textbf{Kernels, cokernels.} Once you have the cohomology functors, you can verify that $P(Y)$ is abelian:
take an arrow $a:P\to Q$ in $P(Y)$, form its cone $C\in D(Y)$, and then you need to verify that $\pc{-1}{C} \to P$
is the kernel and that $Q \to \pc{0}{C}$ is the cokernel. What is the image?

\textbf{Verdier duality exchange}: 
\beq\la{0987}
\pc{i}{C^\vee} =
\pc{-i}{C}^\vee 
\eeq
If $C \cong C^\vee$, then $(\pc{i}{C})^\vee =
\pc{-i}{C}$ and if, in addition, $f$ is proper, then  
\beq\la{vdrf}(Rf_* C)^\vee \cong Rf_* C.
\eeq
The perverse cohomology sheaves of a complex do not determine 
the complex.  However,
Exercise~\ref{okn} tells us that in the decomposition theorem  (\ref{eq2}) we may write 
 \[
 Rf_* \m{IC}_X(M )\cong \bigoplus_{c\in \zed} \pc{c}{Rf_* \m{IC}_X(M)} [-c].
 \]

The perverse cohomology sheaf construction is a way to get perverse sheaves out of any complex. So there are plenty of perverse sheaves.
In fact, we have the following two  rather deep and very (!) surprising facts:

\begin{tm}\la{beno} {\rm (\textbf{The constructible derived category as a derived category})}
Let $Y$ be a variety. 
\ben
\item
$D(Y,\zed)$ with its standard t-structure is equivalent to $D^b(Sh_c(Y,\zed))$ with its standard t-structure
{\rm (Nori \ci{nori}\footnote{This paper contains a lovely proof of the Artin vanishing in characteristic zero})}. 
\item
$D(Y,\rat))$ with its perverse t-structure is equivalent to $D^b(P(Y,\rat))$ with its standard t-structure
{\rm (\ci{beili})\footnote{Beilinson also proves his wonderful Lemma 3.3, a  strengthening of the Artin vanishing in arbitrary characteristic; Nori calls it the ``Basic Lemma'' in \ci{nori}})}. 
\een
\end{tm}

Exercise~\ref{okmm} introduces another construction that leads to special perverse sheaves, i.e. the intermediate extension functor
$j_{!*}$. This is crucial, in view of the fact that
intersection cohomology complex $IC_S(L)$  can be defined as the intermediate extension of $L[\dim{S}]$ from 
the open subvariety $S^o\subseteq S_{reg}\subseteq S$ on which $L$ is defined to the whole of $S,$ and thus to any
variety that contains   $S$  as a closed  subvariety.

 \subsection{Intersection complexes}\la{intco1}
Recall the conditions of support and co-support for a complex $P$ to be a perverse sheaves:
$\dim supp \, \m{H}^i(C) \leq -i$, and the same for~$C^\vee.$

The original definition of  intersection complex $IC_S(L)$ of an enriched variety (recall that we start with  $L$ a locally constant sheaf on some 
$S^o \subseteq S_{reg} \subseteq S)$ involves repeatedly pushing-forward and standard-truncating
across the strata of a suitable stratification of $S$, starting from $S^o$; see \ci{bbd}, Proposition 2.1.11. It is a fact that shrinking $S^o$
does not effect the end result (this is an excellent exercise).  The end result can be characterized as follows.

\textbf{Conditions of (co)support for intersection complexes.}
The intersection complex $IC_S(L)$ of an enriched variety $(S,L)$ is the complex $C$, unique up to 
unique isomorphism subject to the following conditions of support and co-support: $C_{|S^o} = L[\dim{S}]$
$\dim supp \, \m{H}^i \leq -i-1$ for every $i \neq -\dim{S}$, and the ``same'' for $C^\vee$.
Recall the two vectors exemplifying the conditions of support  in dimension four for perverse sheaves $(4,3,2,1,0)$, and for 
intersection complexes  $(4,2,1,0,0)$. Using this characterization, the reader should verify that direct images
under finite maps preserve intersection complexes and that the same is true for small maps (see \S\ref{lz3}).

\textbf{Another characterization of  intersection complexes.}
The intersection complex $IC_S(L)$ is the unique perverse sheaf
extending its own restriction to an open dense subvariety $U \subseteq S$
so that the extension is ``minimal'' in the following sense: it has no non zero
perverse subobject or quotient supported on the boundary $S\setminus U$\footnote{It may have, however, a non zero subquotient supported on the boundary; the formation of this 
kind of minimal --a.k.a. intermediate extension-- is not exact on the 
relevant abelian categories: it preserves injective and surjective maps, but it does not preserve exact sequences; see \ci{bams}, p.562.}.

\textbf{Intersection complexes as intermediate extensions.} Let $j: U \to S$
be a  locally closed embedding. Let $P \in P(U)$. Take the natural map (forget the supports)
$Rj_!P  \to Rj_*P.$ Take the map induced at the level of $0$-th perverse cohomology sheaves:
 $a: \pc{0}{Rj_! P} \to \pc{0}{Rj_*P}$. Define  the intermediate extension of $P$  on $U$ to $Y$  by setting 
 $j_{!*} P:= \im\, a \in P(S)$. Let $(S,L)$ be an enriched variety and let $j:S^o\to S$;
 we apply the intermediate extension functor  to ${P:= L[\dim S] \in P(U:= S^o)}$ and we end up with  $IC_S(L)!$ The same conclusion holds if we take any $U$ and $P:= (IC_S(L))_{|U})$ (the intersection complex is the 
intermediate extension of its restriction to any dense open subvariety). 

Since $P(Y)$ is Noetherian and closed under Verdier duality, it is artinian, so that the Jordan-Holder theorem holds.
Now it is a good time to carry out 
Exercise~\ref{artinian}, which shows how to produce  an ``explicit'' Jordan-Holder decomposition for perverse sheaves.
The method also makes it clear that the simple objects in $P(Y)$ are the $IC_S(L)$ with $S \subseteq Y$ an irreducible closed subvariety and 
with $L$  a simple locally constant sheaf  on some Zariski dense open subset of the regular part of~$S.$

\subsection{Exercises for Lecture 2}\la{exlz2}

\begin{exe}\la{cantor}{\rm ( \textbf{Some very non constructible sheaves}) Use the closed embedding $i:\mathfrak{C} \to {\bb A}^1$ of the Cantor set into the complex affine line to show
that the direct image sheaf $i_*\rat_\mathfrak{C}$ is not constructible. Classify the sheaves of rational 
vector spaces on   ${\bb A}^1$ which are both constructible and injective.}
\end{exe}

\begin{ex}\la{0101} {\rm (\textbf{First (non) examples})
If $Y$ is of pure dimension and  $F\in Sh_c(Y)$, then $F[\dim{Y}]$ satisfies the conditions
of support. In general,  $F[\dim{Y}]$  is not perverse as its Verdier dual may fail to satisfy the condition of support.
For example, the Verdier dual of $\rat_Y[\dim{Y}]$  is the shifted dualizing complex $\omega_Y[-\dim{Y}]$
and the singularities of $Y$ dictate whether or not it satisfies the conditions of support; see \ci{bamsv1}, \S4.3.5-7. If $Y$ is
nonsingular of pure dimension, and $L$ is locally constant, then $L[\dim{Y}]$ is perverse, for its Verdier dual is
$L^{\vee} [\dim{Y}].$} \end{ex}

\begin{exe}\la{bt5}{\rm (\textbf{Some perverse sheaves})
The derived direct image of a perverse sheaf via a finite map is perverse. Give examples showing 
that  the  derived direct image via a quasi-finite map of a perverse sheaf
may fail to be perverse. 
 Let $j: X:=\comp^* \to \comp=:Y$ be the natural open embedding; show that  the natural map in $D(Y)$  $Rj_! \rat_X[1]
\to Rj_* \rat_X[1]$  is in fact in $P(Y)$; determine kernel, cokernel and image.  
Show that if we replace $\comp^*$ with $\comp^n\setminus \{0\}$, then  the map above is not one of perverse sheaves.
Show that if instead of removing the origin, we remove a finite configuration of hypersurfaces, then we get a map of perverse sheaves.
Show more generally that  an affine open immersion  $j$ is such that $Rj_!$ and $Rj_*$ preserve perverse sheaves 
(hint: push-forward and use freely the Stein -instead of affine- version of the Artin vanishing theorem
to verify the conditions of support on small balls centered at points on the hypersurfaces at the boundary).
Show that the direct image $Rf_*\rat_X[2]$ with $f:X\to Y$ a resolution of singularities of a surface
is perverse. Let $f:X^{2d} \to Y^{2d}$ be proper and birational, with $X$ nonsingular and irreducible, let $y\in Y$ and let $f$ be an isomorphism
over $Y\setminus y$; give an iff condition that ensures that $Rf_*\rat [\dim{X}]$ is perverse.
Determine the pairs $m\leq n$ such that the blowing up of $f: X\to Y$ of  $\comp^m \subseteq \comp^n=:Y$   is such that
$Rf_* \rat_X[n]$ is perverse.}
\end{exe}

\begin{exe}\la{avtpf} {\rm (\textbf{Artin vanishing: from constructible to perverse sheaves; cohomological dimension})
Assume the Artin vanishing theorem for constructible sheaf and deduce the one for the cohomology
of perverse sheaves by use of the Grothendieck spectral sequence $H^p(Y, \m{H}^q(P)) \Longrightarrow H^{p+q}(Y,P)$.
(Hint: the supports of the cohomology sheaves are closed \underline{affine} subvarieties.)
Dualize the result to obtain the Artin vanishing theorem for the cohomology with compact support of a perverse sheaf.
Use a suitable affine covering of a quasi-projective variety $Y$  to show that the cohomology and cohomology with compact supports
of a perverse sheaf live in the interval $[-\dim{Y}, +\dim{Y}].$ What about a non quasi projective $Y$?
}
\end{exe}

\begin{exe}\la{vbgf}{\rm (\textbf{Intersection complex via push-forward and truncation})
The original Goresky-MacPherson's definition of intersection complex   involves repeated  push-forward and truncation across the strata of a Whitney-stratification. Let us take $j: \comp^n\setminus \comp^{m=0}=:U \to Y:= \comp^n.$
The formula reads $IC_{Y} := \td{-1} Rj_* \rat_U[n]$. Verify that the result is $\rat_Y[n]$. Do this for other values of $m$
and verify that you get $\rat_Y[n]$, again. Take a complete flag of linear subspaces in $\comp^n$, apply the 
general formula given by iterated push-forward and truncation, and verify that you get $\rat_Y[n]$.  What is your conclusion? 
}
\end{exe}

\begin{exe}\la{tuch}{\rm (\textbf{Hodge-Lefschetz package for intersection cohomology})
Guess the precise statements of the following  results concerning intersection cohomology groups:
Poincar\'e duality, 
Artin vanishing theorem, Lefschetz hyperplane theorem, existence of pure  and mixed Hodge structures,
hard Lefschetz theorem, primitive Lefschetz decomposition, Hodge-Riemann bilinear relations.
}
\end{exe}

\begin{exe}\la{kecoke}
{\rm
(\textbf{Injective or surjective?})
Let $j: \bb{A}^1\setminus \{0\} \to \bb{A}^1$ be the natural open embedding. Verify that:
the natural map $j_! \rat \to j_*\rat$ in $Sh_c (\bb{A}^1)$ is injective;  the
natural map  $j_![1] \rat \to j_*\rat [1]$ is \ldots surjective in $P (\bb{A}^1)$.
}
\end{exe}

\begin{exe}\la{okmm}
{\rm
(\textbf{Perverse cohomology sheaves and the intermediate extension functor})
For $j: \comp^n\setminus \{0\} \to \comp^n$ the natural open embedding,
compute $\pc{i}{Rj_! \rat}$ and $\pc{i}{Rj_* \rat}$. Let $j:U\to X$ be an open embedding and let $P \in P(U)$.
Let $a:Rj_! P \to Rj_* P$ be the natural map.
Show that the  assignment defined by $P \mapsto j_{!*} P:= \im \{ \pc{0}{Rj_!P} \to \pc{0}{Rj_*P}\}$ is functorial.
This is the intermediate extension functor.
Find an example showing that it is not   exact (exact:= it sends short exact sequences in $P(U)$ into ones in $P(X)$) (hint:
punctured disk and rank two unipotent and non-diagonal  matrices).
Compute $j_{!*} \rat_U[\dim{U}]$ when $j$ is the embedding of a Zariski-dense open subset
of a nonsingular and irreducible variety. Same for the embedding of affine cones  over projective manifolds minus their vertex into
the cone. Compute $j_{!*} L [1]$ where   $U=\comp^*$ and $L$ is a locally constant sheaf on $U.$
}
\end{exe}

\begin{exe}\la{artinian}{\rm 
(\textbf{Jordan-Holder for perverse sheaves}) Do not assume varieties are irreducible, nor pure-dimensional.
Let $P\in P(Y).$  Find a non empty open nonsingular  irreducible subvariety $j:U\subseteq Y$ such that $Q:=j^*P = L [\dim U]$
for a locally constant sheaf on $U.$ Produce the   natural commutative diagram with $a'$ epimorphic 
and $a''$ monomorphic
\[
\xymatrix{ & P \ar[rd] & \\
\pc{0}{Rj_! Q} \ar[rr]^a  \ar[ru] \ar[rd]^{a'}&& \pc{0}{Rj_* Q} \\
& IC_{\overline{U}}(L) :=\im{a'} \ar[ru]^{a''}.&
}
\]
Deduce formally that $P$ contains a subobject $P'$   together with a surjective map $b:P'\to IC_{\overline{U}}(L).$
Deduce that we have a filtration $\ke{\, b} \subseteq P' \subseteq  P$
with $P'/\ke{\, b}=  IC_{\overline{U}}(L).$ Use noetherian induction to prove that
we can refine  this two-step filtration to a filtration with successive quotients of the form
$IC_S(L)$. Each local system $L$ appearing in this way admits a finite filtration with simple quotients. Refine further to obtain a finite increasing filtration
of $P$ with successive quotients of the form  $IC_S(L)$ with $L$ simple. (In the last step you need to use the fact that
the intermediate extension functor, while not exact, preserves injective maps.)
}
\end{exe}

\begin{exe}\la{conecve} {\rm (\textbf{Attaching the vertex to a cone}) Use attaching triangles and resulting long exact sequences
of cohomology to study $C= Rf_*\rat_X,$ where
$f: X\to Y$ is the resolution of the cone (affine and projective) over a nonsingular embedded projective curve  obtained
by blowing up the vertex. (See \ci{decmigleiden}.)}
\end{exe}

\section{Lecture 3: Semismall maps}\la{lz3}

\textbf{Summary of Lecture 3.}  {\em 
Definition of semismall map. Hard Lefschetz and Hodge-Riemann bilinear relations for semismall maps. Special form of the decomposition theorem for semismall maps.
Hilbert schemes of points on smooth surfaces and the  Grojnowski-Nakajima picture. The endomorphism  and correspondence
algebras are isomorphic and semisimple. Hint of a relation to the Springer picture. }

\subsection{Semismall maps}\la{smmps}
Semismall maps are a very special class of maps, e.g. they are necessarily generically finite.
On the other hand, the blowing up of  point in $\comp^3$ is not semismall. Resolutions of singularities
are very rarely semismall (except in dimension $\leq 2$).
It is remarkable that semismall maps appear in important situations, e.g. holomorphic
symplectic contractions, quiver varieties, moduli of bundles on surfaces, Springer resolutions, convolution on affine grassmannians, standard resolutions
of theta divisors, Hilbert-Chow maps for Hilbert schemes of points on surfaces \ldots
References include  \ci{herdlef} and the beautiful book \ci{chrissginzburg}.
 
We now discuss some of their features. To simplify the discussion, we work
with proper surjective maps $f:X \to Y$  with $X$ nonsingular  and irreducible.

We start with what is likely to be the quickest  possible definition of semismall map.

\begin{defi} {\rm (\textbf{Definition of semismall map})
The map $f$ is said to be semismall if $\dim X\times_Y X =\dim{X}.$}
\end{defi}

Quick is good, but not always transparent. The standard  definition involves consideration of the dimension
of  the locally closed loci $S_k \subseteq Y$ where the fibers of the map
have fixed dimension $k$. Then semismallness is the requirement that $\dim S_k + 2k  \leq \dim{X}$ for every $k\geq  0;$ see Exercise~\ref{dfzsemis}.

\textbf{Small maps.}
We say that the map $f$ is small if it is semismall and $X\times_YX$ has a unique irreducible component of maximal dimension $\dim{X}$ (which one?). For semismall maps, this  is equivalent to having 
$\dim S_k + 2k  <\dim{X}$ for every $k> 0.$

The blowing ups of $\comp^m \subseteq \comp^n$, $m\leq n-2$ are semismall iff
$m=n-2.$ None of these is small. 
The blowing up of the affine cone over the nonsingular quadric in $\pn{3}$ along a plane thru
the vertex is a small map. The blowing up of the vertex is not.

The Springer resolution of the nilpotent cone in a semisimple Lie algebra is semismall and Grothendieck-Springer
simultaneous resolution is small. We shall meet both  a bit later and show how they interact beautifully to give us
a ``decomposition theorem argument'' for the presence an action of the Weyl group of the Lie algebra
on the cohomology of the fibers of the Springer resolution. The Weyl group does not act  on the fibers!

The following beautiful result of D. Kaledin's is a source of a large example of highly non trivial
semismall maps.

\begin{tm}\la{kaltm} {\rm (\textbf{Holomorphic symplectic contractions are semismall} \ci{kaledin})}  A projective birational map from a holomorphic symplectic\footnote{I.e. even-dimensional and admitting  a closed  holomorphic $2$-form $\omega$ which is non-degenerate, i.e. $\omega^{\frac{\dim{X}}{2}}$ is nowhere vanishing}  nonsingular variety
is semismall.
\end{tm}

It is amusing to realize that semismallness  and the Hard Lefschetz phenomenon are essentially equivalent. In fact, we have the following

\begin{tm}\la{hlsemi} {\rm (\textbf{Hard Lefschetz for semismall maps} \ci{herdlef})}
\textbf{Hard Lefschetz and semismall maps.}  Let $f: X\to Y$ be a surjective  projective map of projective varieties
with $X$ nonsingular and let $\eta:=f^*L \in H^2(X,\rat)$ be the  first Chern class of the pull-back to $X$  of an ample line bundle 
$L$ on $Y.$
 The iterated cup product maps $\eta^r: H^{\dim{X}-r}(X,\rat) \to H^{\dim{X} +r}(X,\rat)$ are isomorphisms for every $r\geq 0$ iff the map $f$ is semismall. In the semismall case, we have  the primitive Lefschetz decomposition
and the Hodge-Riemann bilinear relations.
\end{tm}

\textbf{Hodge-index theorem for semismall maps.}
There is an important phenomenon concerning projective maps that is worth mentioning, i.e. the signature of certain local intersection forms \ci{decmightam}; for a discussion of these, see \ci{decmigleiden}.
The situation is more transparent in the case of semismall maps, where it
 is directly related to the Hodge-Riemann bilinear relations associated with of Theorem~\ref{hlsemi}.
To have a clearer picture, let us  limit ourselves to state a simple,  revealing   and important special  case.
Let $f:X \to Y$ be a surjective semismall projective map with $X$ nonsingular of some even  dimension $2d$.
Assume that $f^{-1}(y)$ is $d$-dimensional for some $y \in Y.$ By intersecting in
$X$ we obtain the refined  symmetric intersection pairing $H_{2d} (f^{-1}(y)) \times
H_{2d} (f^{-1}(y)) \to \rat$, where we are intersecting the fundamental classes of the irreducible components
of top-dimension $d$
of this special fiber inside  of
$X$. The following is a generalization of a result
of a famous result of Grauert's for $d=1.$

\begin{tm}\la{grd}{\rm (\textbf{Refined intersection forms have a precise sign})}
The refined intersection pairing above  is $(-1)^d$-positive-definite.
\end{tm}

By looking carefully at every proper map, similar refined intersection forms
reveal themselves. One can prove that the DT is (essentially)  equivalent to the non-degeneracy
of these refined intersection forms together with Deligne's semisimplicity of monodromy theorem; see \ci{decmightam}.

Exercise~\ref{semiperv} relates perverse sheaves and semismall maps:
if $f: X\to Y$ is semismall, then $Rf_* \rat_X[\dim{X}]$ is perverse.

The decomposition theorem  then tells us that 
\beq\la{dtsemismall}
f_* \rat_X [\dim X] = \pc{0}{f_* \rat_X [\dim X]} = 
\bigoplus_{(S,L) \in EV_0} IC_S(L).
\eeq

\begin{??}\la{domsm}{\rm 
What ``are'' the summands  appearing in the  decomposition theorem 
for semismall maps?}
\end{??}

\subsection{The decomposition theorem for semismall maps}\la{endo}

\textbf{A little  bit about stratifications.}
Even if not logically necessary, it simplifies matters to use the stratification theory of maps
to clarify the picture a bit.  There is a finite disjoint union decomposition
$Y=\coprod_{a \in A} S_a$ into locally closed nonsingular irreducible  subvarieties $S_a \subseteq Y$
such that $f^{-1}(S_a) \to S_a$ is locally (for the classical topology) topologically a product over $S_a.$
It is clear that this decomposition refines the one above given by the dimension of the fibers,
so that $\dim S_a + 2 \dim f^{-1} (s) \leq \dim {X}$ for every $s \in S_a.$ Since the map is assumed to be proper,
it is also clear that all direct image sheaves $\sq$ restrict to locally constant sheaves on every $S_a$. We call the $S_a$
the strata (of a stratification of the map $f$).

\begin{defi}\la{relstr} {\rm (\textbf{Relevant stratum})
We say that $S_a$ is relevant if we have: \[
\dim S_a + 2 \dim f^{-1} (s) = \dim {X}.\]
We denote by $A_{rel} \subseteq A$ the set of relevant strata.
}
\end{defi}

Exercise~\ref{corelt} shows that, for each relevant stratum $S_a$,  the direct image sheaf
$\m{R}^{\dim{X} - \dim{S_a}}$ restricted to $S_a$ is locally constant, semisimple, with finite monodromy.
We denote this restriction by $L_a$.

\begin{tm}\la{dtsemi} {\rm (\textbf{Decomposition theorem  for semismall maps})}
 Let $f:X\to Y$ be proper surjective  semismall with 
$X$ nonsingular. Then there is a direct sum decomposition
\[f_* \rat_X[\dim{X}] = \bigoplus_{a \in A_{rel}} IC_{\overline{S_a}} (L_a).\]
\end{tm}

Since the locally constant sheaf $L_a$ is semisimple,  
 it admits the isotypical direct sum decomposition (\ref{isotyp}), i.e. we have   $L_a = \oplus_\chi L_{a,\chi} \otimes M_{a,\chi}$
where $\chi$ ranges over a finite set of distinct isomorphism classes 
of simple locally constant sheaves on $S_a$ and $M_{a,\chi}$ is a vector space of rank
the  multiplicity $m_{a,\chi}$ of the  locally constant sheaf $L_\chi$ in $L_a.$
The decomposition theorem  then reads 
\beq\la{nju}
f_* \rat_X [\dim{X}] =
\bigoplus_{a , \chi} IC_{\overline{S_a}} (L_{a,\chi} \otimes M_{a,\chi}).
\eeq

\subsection{Hilbert schemes of points on surfaces and Heisenberg algebras}\la{hilb}

An excellent reference is \ci{nakahilb}.
Let $X$ be a nonsingular  complex surface. For every $n \geq 0$ we have the Hilbert scheme $X^{[n]}$
of $n$ points on $X$.   It is irreducible nonsingular of dimension $2n$. There is proper birational
surjective map  $\pi: X^{[n]} \to X^{(n)}$ onto the $n$-th symmetric product sending a length $n$ zero dimensional subscheme
of $X$ 
to its support counting multiplicities. There is a natural stratification of the symmetric product
 of the map: $X^{(n)} = \coprod_{\nu \in P(n)} X^{(n)}_\nu$, where $P(n)$ is the set of partitions $\nu=\{\nu_j\}$
 of the integer $n$ (the $\nu_j$ are positive integers adding up to $n$) obtained by taking the locally closed irreducible
 nonsingular dimension $2 l(\nu)$  ($l$ is the length of partition  function)
 sets of points $\sum_i n_ix_i \in X^{(n)}$ with type $\{n_i\}$ given by $\nu.$ The remarkable fact is 
 that the fibers $\pi^{-1} (x_\nu)$ of the points $x_\nu \in X^{(n)}_{\nu}$ are irreducible
 of dimension $\sum_j (\nu_j -1) = n- l(\nu).$ It follows that the map $\pi$ is semismall. In fact, $X^{[n]} = \coprod_{P(n)}
 \pi^{-1} (X^{(n)}_{\nu}) \to \coprod_{P(n)} X^{(n)}_{\nu}$ is a stratification of the  semismall map $\pi$ and all the 
 strata are relevant. Since the fibers are all irreducible, the relevant locally constant sheaves are all constant of rank one.
 In particular, the decomposition theorem for $\pi$ takes the form
\[R\pi_* \rat_{X^{[n]}} [2n] = \oplus_{\nu \in P(n)} IC_{\overline{X^{(n)}_{\nu}}}\,.\]
 A second remarkable fact is that the normalization $\overline{X^{(n)}_{\nu}}$ can be identified
 with a product of symmetric products $X^{(\nu)}:= \prod_{i=1}^{n} X^{(a_i)}$, where $a_i$ is the number of times
 that $i$ appears in $\nu.$ This is a variety obtained by dividing a nonsingular variety by the action of a finite group; in particular,
 its intersection complex is the constant sheaf.  By the IC normalization principle Fact~\ref{nocpri}, we see that
 $IC_{\overline{X^{(n)}_{\nu}}}$ is the push-forward of the  shifted constant sheaf from the normalization.
 By taking care of shifts, after the dust settles, we obtain the G\"ottsche formula: (for $n=0$ take $\rat$ on both sides)
 \beq\la{gofo}
 {\bb H}(X):=
 \bigoplus_{n\geq 0} H^*(X^{[n]}) = \bigoplus_{n\geq 0} \bigoplus_{\nu \in P(n)}  H^{*- cl(\nu)} (X^{(\nu)}, \rat),
 \eeq where the colength $cl(\nu):= n -l(\nu).$

If we take $X = \comp^2$, then something remarkable emerges: look at (\ref{23}) and (\ref{24}).  The formula above, 
 taken for every $n \geq 0$ gives
\beq\la{23}
\sum_{n=0}^\infty \dim H^*({\comp^2}^{[n]}, \rat) = \prod_{j=1}^\infty \frac{1}{1-q^j}.
\eeq
Let $R:= \rat [x_1, x_2, \ldots ]$ be the algebra of polynomials in the  infinitely many indeterminates $x_i$, declared to be of degree $i.$
The infinite dimensional Heisenberg algebra $\ms{H}$ is the Lie algebra
with underlying rational vector space the one with basis $\{ \{ d_i \}_{i<0}, c_0, \{m_i\}_{i>0} \}$
and subject to the following relations: $c_0$ is central, the $d_i$'s commute with each other,
the $m_i$'s commute with each other, and $[d_{i}, m_j]=\delta_{-i,j} c_0$. Then $R$ is an irreducible
$\ms{H}$-module generated by $1$ where $d_i$ acts as formal  derivation by $x_i$ and $m_j$ by multiplication by $x_j$.
The dimension $a_n$  of the space of homogeneous polynomials of degree $n$ is given by:
\beq\la{24}
\sum_{n=0}^\infty a_n q^n = \prod_{j=1}^\infty \frac{1}{1-q^j}.
\eeq
Well, isn't this a coincidence! The operators $d_i, m_i$  change the  homogeneous degree of $x$-monomials by $\pm i$. This,
together with the formalism of correspondences in products,  suggests
that  there should be geometrically meaningful cohomology classes in $H^*({\comp^2}^{[n]} \times {\comp^2}^{[n\pm i]}, \rat)$
that reflect, on the Hilbert scheme side, the   Heisenberg algebra action
on the polynomial side. 

This is indeed the case and it is due to Grojnowski and  to Nakajima: they  guessed what above, constructed algebraic cycles on the 
products of Hilbert schemes above that would be good candidates and then verified the Heisenberg Lie algebra relations.
In fact,  for every nonsingular surface $X$,  there is an associated  (Heisenberg-Clifford) ``algebra''  $\ms{H}(X)$ that acts geometrically and irreducibly
on $\bb{H}(X)$ (\ref{gofo}).

\subsection{The endomorphism algebra \texorpdfstring{$\text{End} (f_*\rat_X)$}{End (f*QX)}}\la{endo1}
An reference here is \ci{decmigsemi}.

\textbf{Semisimple algebras.}  A semisimple algebra is an  associative artinian (dcc) algebra 
over a field with trivial Jacobson ideal (the ideal killing all simple left modules). The Artin-Wedderburn theorem classifies the semisimple algebras over a field as the ones which are finite Cartesian products
of matrix algebras over  finite dimensional division algebras over the field. 

\textbf{Warm-up.}
Show that
$M_{d \times d}(\rat)$ is semisimple. Show that the upper triangular matrices do not
	form a semisimple algebra. Hence if ${f:= pr_2: \pn{1} \times \pn{1} \to \pn{1}}$,
then $\text{End}_{D(\pn{1})} (Rf_* \rat)$ is not a semisimple algebra.

\begin{tm}\la{semsemis}{\rm (\textbf{Semismall maps and semisimplicity of the End-algebra})}  Let things be as in Theorem~\ref{dtsemi}.
The endomorphism $\rat$-algebra $\text{End}_{D(Y)} f_* \rat_X [\dim{X}]$
is  semisimple.
\end{tm}
\begin{proof} We see more important properties of intersection complexes at play,
i.e. the Schur lemma phenomena for simple perverse sheaves.

\n
By simplicity, $Hom (IC_{\ov{S_a}} (L_{\chi}), IC_{\ov{S_b}} (L_{\psi}) ) = 
\delta_{\chi, \psi} \text{End}(IC_{\ov{S_a}}(L_{\chi}))$ (i.e., there are no non zero maps if they differ): in fact,
look at kernel and cokernel  and use simplicity.  This leaves us with considering
terms of the form $\text{End}(IC_{\ov{S_a}}(L_{\chi}))$ whose elements, for the same reason as above,
are either  zero,  or are  isomorphisms. These terms are thus division algebras
$D_{a,\chi}$.  It follows that \[\text{End}_{D(Y)} (Rf_* \rat_X [\dim{X}]) = \prod_{a,\chi} 
M_{d_{a,\chi} \times d_{a,\chi}} (D_{a,\chi}),\]
which is a semisimple algebra by the Artin-Wedderburn theorem.
\end{proof}

\textbf{The endomorphism algebra as a geometric convolution algebra.}
A reference here is also  \ci{chrissginzburg}.
We can realize the algebra $\text{End}_{D(Y)} (Rf_*\rat_X)$  of endomorphisms  in the derived category  in geometric terms as the convolution algebra
${H_{2\dim{X}}^{BM} (X\times_Y X)}$, which is thus semisimple. Let us discuss this a bit.

Let $X$ be a nonsingular projective variety. Then, we have an isomorphism of
algebras between the first and last term
\beq\la{cosio}
\begin{aligned}
\text{End}(H^*(X)) &=_1 H^*(X)^\vee \otimes H^*(X) \\
&\cong_2 
H^*(X)\otimes H^*(X) \cong_3 H^*(X \times X) \cong_4 H_*(X\times X)
\end{aligned}\eeq
where: $=_1$: linear algebra; $\cong_2$:
Poincar\'e duality; $\cong_3$: K\"unneth; $\cong_4$: Poincar\'e duality; and where:
the algebra structure on the last term is the one given by the formalism of composition
of  correspondences in products (see Exercise~\ref{corr}).

The classes $\Gamma \in H_{\gamma} (X\times X)$ appearing in Exercise~\ref{corr} are called correspondences. This picture generalizes well, 
but not trivially, to  proper maps $f: X \to Y$ from nonsingular varieties as follows.

\begin{tm}\la{cons} {\rm (\textbf{Correspondences and endomaps in the derived category})} Let $f: X \to Y$ be a proper 
map from a nonsingular variety.
There is a natural isomorphism  $\text{End}_{D(Y)} (f_*\rat_X) \cong H^{BM}_{2\dim{X}} (X\times_Y X)$ of $\rat$-algebras.
\end{tm}

For our semismall maps, Exercise~\ref{5050} provides  an evident geometric basis of the vector space ${H^{BM}_{2 \dim X} (X \times_Y X)}$.

Since there is a basis of $H^{BM}_{2 \dim X} (X \times_Y X) =\text{End}_{D(Y)} (Rf_* \rat_X)$ given by algebraic
cycles, a formal linear algebra manipulation shows that if $X$ is projective, then  
decomposition  $H^*(X,\rat)=
\oplus_{a \in A_{rel}} IH^{*- {\rm codim}(S_a)}(\overline{S_a}, L_a)$ is compatible with the Hodge $(p,q)$-decomposition,
i.e. it is given by pure Hodge substructures; see Exercise~\ref{iths}.  In fact, one even has a canonical decomposition of Chow motives
reflecting the decomposition theorem for semismall maps; see \ci{decmigsemi}.  Look at the related (deeper) Question~\ref{modt}.

\subsection{Geometric realization of the representations of the Weyl group}\la{georw}
An excellent reference is \ci{chrissginzburg}.

There is a well-developed theory of representations of finite groups $G$ (character theory)
into finite dimensional complex vector spaces. In a nearly tautological sense, this theory is equivalent to
the representation theory of the group algebra~$\rat [G]$.

If we take the Weyl group  $W$ of any  of the usual suspects, e.g. $SL_n(\comp)$ with Weyl group
the symmetric group $S_n$, then  we can  ask whether
we can realize  the irreducible representations of $W$ by using the fact that $W$ is a Weyl group.

Springer realized that this was indeed possible, and in geometric terms!
In what follows, we do not reproduce this amazing story, but we limit ourselves to showing how the decomposition theorem allows
(there are other ways) to introduce the action of the Weyl group
on the cohomology of the Springer fibers. The Weyl group does not act  algebraically on these fibers!

Take the Lie algebra  $sl_n (\comp)$ of traceless $n\times n$ matrices. Inside of it there is
the cone $N$ with vertex the origin given by the nilpotent matrices. Take the flag variety $F$,
i.e. the space of complete flags $f$ in $\comp^n$. Set $\w{N} := \{(n,f) |, n \, \mbox{stabilizes} \, f\} \subseteq N\times F.$
Then $\w{N} \to F$ can be shown to be the projection $T^*F \to F$ for the cotangent bundle (and $\w{N}$ is thus  a holomorphic symplectic manifold)
and  the projection $\pi: \w{N} \to N$ is a resolution of the singularities on the nilpotent cone $N$.

The map $\pi$ is semismall! We know this, for example, from Kaledin's Theorem~\ref{kaltm}. In fact, it was known much earlier,
by the work of many.
We can partition $N$ according to the Jordan canonical form. This gives rise to a stratification of $N$ and of $\pi.$
Every stratum of this stratification turns out to be  relevant for the semismall map~$\pi.$ 

It is amusing to realize that the intersection form associated with the deepest stratum (vertex)  gives rise to $\pm$ the Euler number of
$F$ and that the one associated with the codimension two stratum yields the Cartan matrix
for~$sl_n(\comp)$.

The fibers of the map $\pi$ are called Springer fibers. Springer proved that all the irreducible representations
of the Weyl group occur as direct summands of the action of the Weyl group on the  homology of the Springer fibers.
This beautiful result tells us that indeed one can realize geometrically such representations.

Note that the Weyl group does not act algebraically on the Springer fibers. 

In what follows we aim at  explaining how the Weyl group acts on the perverse sheaf $R\pi_* \rat_{\w{N}}.$
In turn, by taking stalks, this explains why the Weyl group acts on the homology of the Springer fibers.

Instead of sticking with $\pi: \w{N} \to N$, we consider $p:\w{sl_n(\comp)} \to sl_n(\comp)$ defined by the same kind of incidence relation.
The difference is that if we take the Zariski open set $U$ given by diagonalizable matrices with $n$ distinct eigenvalues, 
then $p$ is a topological Galois cover with group the Weyl group (permutation of eigenvalues).
We have a Weyl group action! Unfortunately $N \cap U = \emptyset!$ On the other hand, $p$ is \ldots small! So $Rp_*
\rat=\m{IC}(L)$, where $L$ is the local system on $U$ associated with the Galois cover with group
the Weyl group. Then $W$ acts on $L.$ Hence it acts on $IC(L)$ by functoriality of the intermediate extension
construction. Since the map $p$ is proper, and it restricts to $\pi$ over $N$, we see that
the restriction of $Rp_*\rat = IC(L)$  to $N$ is $R\pi_*\rat$ which thus finds itself endowed, almost by the trick of a magician,
with the desired $W$ action!

\subsection{Exercises for Lecture 3}\la{exlz3}

\begin{exe}\la{dfzsemis}{\rm (\textbf{Semismallness and fibers})}  {\rm Show that, for any maps ${f:X\to  Y}$,  we always have $\dim{X\times_Y X} \geq \dim{X}.$
Use Chevalley's result on the upper semicontinuity of the dimensions of the fibers
of maps of algebraic varieties to produce a finite disjoint union decomposition
$Y = \coprod_{k \geq 0} S_k$ into locally closed subvarieties with 
$\dim f^{-1}(y) =k$ for every $y \in S_k.$ Show that $f$ is semismall
iff we have $\dim S_k + 2k  \leq \dim{X}$\footnote{Think of it as a vary special upper bound on the  dimension of the ``stratum''
where the fibers are $k$-dimensional} for every $k\geq 0.$
Observe that $f$ semismall implies that $f$ is generically finite, i.e. that $S_0$ is open and dense. Observe that $f^{-1}(S_0)\times_{S_0}
f^{-1} (S_0)$ has dimension $\dim{X}$. Give examples of semismall maps where $X\times_Y X$
has at least two irreducible components of dimension~$\dim{X}$.}
\end{exe}

\begin{exe}\la{semiperv} {\rm (\textbf{Semismal maps and perverse sheaves})}
 {\rm A proper map ${f:X\to  Y}$ with $X$ nonsingular  is semismall iff $f_* \rat_X[\dim{X}]$ is perverse on~$Y$.}
 \end{exe}

\begin{exe}\la{corelt} {\rm (\textbf{Relevant locally constant sheaves}) Let $S_a \in A_{rel}$ be relevant. Show that $R^{\dim{X} - \dim{S_a}}$ is locally constant on $S_a$
with stalks $H^{2 \dim f^{-1} (s)} (f^{-1}(s))$. The monodromy of this locally constant sheaf,
denoted by $L_a$,
factors through the finite group of symmetries of the set of
irreducible components of maximal dimension $\frac{1}{2}(\dim{X} - \dim{S_a})$
of a typical fiber $f^{-1}(s).$ (Note that, a priori, the monodromy could
send the fundamental class of such a component to minus itself, thus contradicting
the claim just made; that this is not the case follows, for example, from a theorem  of Grothendieck's  in EGA IV, 15.6.4;
 see the nice general  discussion in B.C. Ngo's paper arxiv0801.0446v3, \S7.1.1.)
 In particular, $L_a$ is semisimple.  Note that
if we switch from $\rat$-coefficients to $\overline{\rat}$-coefficients simple objects
may split further. Do they stay semisimple?}
\end{exe}

\begin{exe}\la{corr}{\rm (\textbf{Formalism of correspondences in products})} {\rm
 Unwind the isomorphisms (\ref{cosio}) to deduce
that under them,  a class $\Gamma \in H_\gamma(X\times X)$
defines  a linear map $\Gamma_*: H^{*}(X) \to H^{*+\gamma - 2\dim{X}}(X)$
given by ${a \mapsto PD({pr_2}_*( pr_1^* a \cap \Gamma))}$. Conversely, show that any graded linear map
$H^*(X) \to H^{*+\gamma - 2\dim{X}}(X)$ is given by a unique such $\Gamma \in H_{\gamma}(X\times X).$}
\end{exe}

\begin{exe}{\rm(\textbf{Sheaf theoretic definition of Borel-Moore homology})
Let $X$ be a variety. Recall (one) definition of Borel-Moore homology: embed $X$ as a closed subvariety of smooth variety $Y$, then set
\[ H^{BM}_i (X) := H^{2\mathrm{dim}\, Y -i}(Y, Y -X), \]
where the right-hand-side is relative cohomology. By interpreting relative cohomology sheaf theoretically, give a sheaf theoretic definition of Bore-Moore homology. 

\n
(Hint: consider the distinguished triangle $i_*i^! \to \mathrm{id} \to Rj_*j^* \to$, if you apply this to the constant sheaf and take cohomology (or equivalently push to a point), what long exact sequence do you get?)
}
\end{exe}

\begin{exe}{\rm(\textbf{Borel-Moore homology and Ext/convolution algebras})
Fix a pro\-per morphism of varieties $\pi\colon X \to Y$, with $X$ smooth. Form a cartesian square
\[ \xymatrix{
Z \ar[r]^-{p_1}\ar[d]_-{p_2} & X\ar[d]^-{\pi} \\
X\ar[r]^-{\pi} & Y
}\]
For sheaves $A, B$ on a variety $Z$, let $Ext^i(A,B) = Hom_{D^b(Z)}(A,B[i])$. I.e., $Ext^{*}$ denotes (shifted) $Hom$ in the derived category. Show the following:
\begin{enumerate}
\item $Ext^{*}(R\pi_* \mathbb{Q}, R\pi_*\mathbb{Q}) = Ext^{*}(\mathbb{Q}, \pi^! R\pi_*\mathbb{Q})$. (Hint: adjunction property).
\item $Ext^{*}(\mathbb{Q}, \pi^! R\pi_*\mathbb{Q}) = Ext^{*}(\mathbb{Q}, Rp_{2*}p_1^! \mathbb{Q})$. (Hint:  proper base change).
\item $Ext^{*}(\mathbb{Q}, Rp_{2*}p_1^! \mathbb{Q}) = Ext^{*}(\mathbb{Q}, p_1^!\mathbb{Q})$ (Hint: push-forward and hom).
\item $Ext^{*}(\mathbb{Q}, p_1^!\mathbb{Q}) = H^{BM}_{2\mathrm{dim}\, X - *}(Z)$. (Hint: use the sheaf-theoretic definition of Borel-Moore homology; the dimensional shift suggests the use of some kind of duality).
\end{enumerate}
}
\end{exe}

\begin{exe}\la{5050} {\rm (\textbf{Geometric basis for $H^{BM}_{2 \dim X} (X \times_Y X)$ when $f$ is semismall}) 
Show that if   $f$ is semismall, then the rational vector space $H^{BM}_{2 \dim X} (X \times_Y X)$ has a basis
formed by the fundamental classes of the irreducible components of $X \times_Y X$ of maximal dimension
$ \dim X$. Describe these irreducible components in terms of monodromy over the relevant strata.}
\end{exe}

\begin{exe}\la{iths} {\rm (\textbf{Hodge-theoretic decomposition theorem for $f$ semismall})
 Show that  if $f: X \to Y$ is semismall
with $X$ projective nonsingular, then the decomposition $H^*(X,\rat)=
\oplus_{a \in A_{rel}} IH^{*- {\rm codim}(S_a)}(\overline{S_a}, L_a)$ is one pure Hodge structures (i.e. compatible
with the Hodge $(p,q)$-decomposition of $H^*(X,\comp).$ (Hint: the projectors
onto the direct sums are given by algebraic cycles.)
}
\end{exe}

\section{Lecture 4: DT symmetries: VD, RHL. \texorpdfstring{$\m{IC}$}{IC} splits off}\la{symm}

\textbf{Summary of Lecture 4.}  {\em Discussion of the two main symmetries in the decomposition theorem
for projective maps: Ver\-dier duality and the relative hard Lefschetz theorem. Hard Lefschetz in intersection cohomology and
Stanley's theorem for  rational simplicial polytopes. A proof that the  intersection complex of the image is always a direct summand.
Pure Hodge structures on the intersection cohomology of a  projective surface.}

\begin{rmk}\la{manc}{\rm
We are going to discuss two symmetries for projective  maps: Ver\-dier duality and the relative Hard Lefschetz
theorem. Both these statements have to do with  the  direct image perverse sheaves. In fact, if the target is projective, then
we can take the shadow of these two symmetries in cohomology and notice that there are two additional symmetries: the Verdier duality
and Hard Lefschetz theorem on the individual summands $IH^*(S,L)$. 
Exercise~\ref{moresymme} asks you to make an explicit list in a low-dimensional case.
}
\end{rmk}

\subsection{Verdier duality and the decomposition theorem}\la{pvd}

\textbf{Verdier duality and the decomposition theorem.} 
Recall the statement  (\ref{eq2}) of the decomposition theorem
\[
Rf_* \m{IC}_X(M) \cong \bigoplus_{q\geq 0, \m{EV}_q} \m{IC}_S(L) [-q].
\]

\textbf{Switching to the perverse intersection complex.}
If $(S,L)$ is an enriched variety, then  irreducible, then  $\m{IC}_S (L)$ is not  a perverse sheaf. Recall (\ref{icvsic}):
the perverse object is $IC_S(L):= \m{IC}_S (L) [\dim{S}]$. In order to emphasize better  certain symmetric aspects of the decomposition theorem, we switch to
the perverse intersection complex.
This entails a   a minor headache when re-writing  (\ref{eq2}), which becomes (verify it as an exercise)
\beq\la{eq2bis}
Rf_* IC_X(M) \cong \bigoplus_{b\in \zed} \bigoplus_{EV_b} {IC}_S(L) [-b],\,\,
\eeq
where $(S,L) \in EV_b$ iff $(S,L) \in \m{EV}_{b+\dim{X} -\dim{S}}$.

\textbf{The perverse cohomology sheaves of the derived direct image.} Note that, now, 
every $b$-th direct summand above is a perverse sheaf, so that, in  view of Exercise~\ref{okn}, we have  that:
\beq\la{vbnt}
\pc{b}{Rf_* IC_X(M)} = \oplus_{EV_b} IC_S(L).
\eeq

\textbf{The case when $M$ is self-dual (and semisimple).}
Let $M$ be self-dual, e.g. a constant sheaf, a polarizable variation of pure Hodge structures, or even the direct sum of any $M$ with its dual; self-dual local systems appear frequently in complex algebraic geometry.
Then so  is  $IC_X(M)$ and, by  the duality exchange property for proper maps, so is $Rf_* IC_X(M).$
In view of the duality relation (\ref{0987}) between  perverse cohomology sheaves,  we see that
$\pc{b}{Rf_* IC_X(M)}) \cong \pc{-b}{Rf_* IC_X(M)})^\vee.$ By combining with (\ref{vbnt}), we get
\beq\la{mki9}
Rf_* IC_X(M) \cong 
  \biggl(\bigoplus_{\substack{b<0\\[2pt]EV_{b}}} IC_S(L) [-b]\biggr)
 \oplus
\biggl(\bigoplus_{EV_{0}} IC_S(L) \biggr) \oplus \biggl(\bigoplus_{\substack{b<0\\[2pt]EV_{b}}} IC_S(L^\vee) [b]\biggr).
\eeq
In other words, the direct image is palindromic, i.e. it reads the same, up to shifts and dualities,  from right to left and from left to right. Just like the cohomology of a compact oriented manifold!

\textbf{The defect of semismallness.}
When trying to determine the precise shape of the decomposition theorem, one important
invariant is the minimal  interval $[-r,r]$ out of which the perverse cohomology sheaves are zero.
In this direction, we have that if $M$ is constant and $X$ is nonsingular, then $r=\dim X\times_Y X - \dim X\geq 0$.
This difference is called the {\em defect of semismallness} in \ci{decmightam}. In this situation, $r=0$ iff the map is semismall.

\subsection{Verdier duality and the decomposition theorem with large fibers}\la{vdlf}
Here is a nice  consequence of Verdier duality, more precisely of (\ref{mki9}). It is an observation due to Goresky and MacPherson and it is used
by B.C. Ng\^{o} in his proof of the support theorem, a key technical and geometric result in his proof
of the fundamental lemma in the Langlands' program. See \ci{ngo}, \S7.3.

\begin{tm}\la{gt6} 
Let $f: X\to Y$ be proper with $X$ nonsingular and equidimensional fibers of dimension $d$. Assume a subvariety
$S$ appears in the decomposition theorem
(\ref{eq2}) for $Rf_*\rat_X.$ Then ${\rm codim} (S) \leq d.$
\end{tm}
\begin{proof}
There is a maximum index $b^+_S \in \zed$ for which  a term $IC_S(L)[-b^+_S]$ appears. By the palindromicity
(\ref{mki9}), we may assume that $b^+_S \geq 0.$ Recall that $L$ is defined on some open
dense $S^o \subseteq S.$ Let $U\subseteq Y$ be open such that its trace on $S$ is $S^o.$
Replace $Y$ with $U$. Denote by $i:S^o\to Y$ the closed embedding. Then $Rf_*\rat [\dim X]$ admits $i_*L[\dim{S}] [-b^+_S]$ as a direct summand.
Then $i_* L$ is a non trivial direct summand of  $R^{\dim{X} -\dim {S} +b^+_S} f_* \rat.$ Since the fibers have dimension $d$ and $b^+_S\geq 0,$ 
we have tha
$\dim{X} -\dim{S} \leq \dim{X}-\dim{S} +b^+_S \leq 2d.$ Since $\dim{X} = \dim{Y} +d$, the conclusion follows.
\end{proof}

\subsection{The relative hard Lefschetz theorem}\la{rhls}

\textbf{Poincar\'e duality vs. hard Lefschetz.}
Let $X$ be a projective manifold and let $\eta \in H^2(X,\rat)$ be the class of a hyperplane section.
Then we have two separate phenomena: 
\[
H^{\dim{X} -r} (X,\rat) = H^{\dim{X}+r}(X,\rat)^\vee  \quad \mbox{(Poincar\'e duality)},
\] 
\[\eta^{r}: H^{\dim{X} -r} (X,\rat) \cong H^{\dim{X}+r}(X,\rat)^\vee
\mbox{(hard  Lefschetz)}.\]
The first statement is that the pairing $\int_X -\wedge -$ between cohomology in complementary degrees is non-degenerate.
The second one is that the pairing $\int_X \eta^r -\wedge -$ on  a cohomology group $H^{\dim{X}-r}$  is non-degenerate.
They both imply the usual
symmetry of Betti numbers. The latter implies  also their unimodality, i.e. that
$b_{d} \geq b_{d-2} \geq \ldots$, where $d$=$ is either $$\dim{X}$ or$ \dim{X}-1$  (compare with the Hopf surface, which has the former, but not the latter).

Exercise~\ref{hl} discusses two proofs of the classical hard Lefschetz theorem. Both proofs generalize and, with some work, afford proofs of the relative hard Lefschetz theorem. In particular, they yield proofs of the hard Lefschetz theorem in intersection cohomology.

We have mentioned how, in the context of singular varieties,  Poincar\'e duality in cohomology is lost but
found again in intersection cohomology.  The same is true for the hard Lefschetz theorem for the intersection cohomology of 
projective varieties! This brings us back to the
theme harping the importance of the derived category and of perverse sheaves: the statement of the hard Lefschetz for intersection cohomology is cohomological, but there is no known proof that avoids perverse sheaves.

Let $f: X \to Y$ be a map of varieties and let $\eta \in H^2(X,\rat)$ be a cohomology class.
It is a general fact that 
 $\eta$ induces, $\eta: C \to C[2]$,
which induces  $\eta: Rf_*C
\to Rf_*C[2]$, which, by taking perverse cohomology sheaves,   induces $\eta: \pc{b}{C} \to \pc{b+2}{C}$.
It follows that, for every $b \geq 0$,  we obtain maps $\eta^{b}: \pc{-b}{C} \to \pc{b}{C}$ in $P(Y)$.

Let  $f:X \to Y$ be a projective map with  and let $IC_X(M) \in P(X)$. Let  ${\eta \in H^2(X,\rat)}$ be the first Chern class of
an $f$-ample line bundle, i.e.  a line bundle
on $X$ which is ample on every fiber of $f$. It is a general fact that 
 $\eta \in H^2(X,\rat)$ induces, $\eta: IC_X(M)
\to IC_X(M)[2]$, which induces  $\eta: Rf_*IC_X(M)
\to Rf_*IC_X(M)[2]$, which, by taking perverse cohomology sheaves,   induces $\eta: \pcs^b \to \pcs^{b+2}$.
It follows that, for every $b \geq 0$,  we obtain maps $\eta^{b}: \pcs^{-b} \to \pcs^b$ in~$P(Y)$.

\begin{tm}\la{rhl}{\rm (\textbf{Relative hard Lefschetz} \cites{bbd,samhm,decmightam,sabbah,mochizuki})} 
Let $f: X \to Y$  be a projective map and let $\eta \in H^2(X,\rat)$ be the first Chern class of an $f$-ample
line bundle on $X$\footnote{Here is one such line bundle:
since $f$ is projective, there is a factorization of $f$ as follows $X\to Y \times  \bb{P} \to  Y$ (closed embedding, followed by the projection); pull-back the hyperplane bundle
from $\bb{P}$ to $Y \times \bb{P}$ and restrict to $X$}.  Let  $IC_X(M)$ be  semisimple, i.e. $X$ irreducible and $M$ semisimple. For every $b\geq 0$ the iterated cup product map
$\eta^b: \pc{-b}{Rf_* IC_X(M)} \to \pc{b}{Rf_* IC_X(M)}$ is an isomorphism.
\end{tm}

\textbf{Hard Lefschetz for the intersection cohomology of projective varieties.}
The special case when $Y$ is a point  and $M$ is constant yields the hard Lefschetz theorem for the intersection cohomology
groups of a projective variety.

The same Deligne-Lefschetz criterion  employed in the proof of the derived Deligne theorem, allows to deduce
formally a first approximation to the decomposition theorem (this argument does not afford  the semisimplicity
part of the decomposition  theorem)
\beq\la{ghjuy}
Rf_* IC_X(M) \cong \bigoplus_b \pc{b}{Rf_* IC_X(M)}[-b].
\eeq

Exercise~\ref{pld} gets you a bit more acquainted with the primitive Lefschetz decompositions.
Exercise~\ref{hl} draws a parallel between the classical inductive approach to the Hard Lefschetz theorem
 via the Lefschetz hyperplane section theorem and the semisimplicity of monodromy
for the family of hyperplane sections; see Deligne's second paper on the Weil Conjectures \ci{weil2}, \S4.1.

\subsection{Application of RHL: Stanley's theorem}\la{fsex}
An excellent reference is \ci{stanley}. For  more details, see \ci{bams}.

A convex polytope is the convex hull of a finite set in real Euclidean space. It is said to be simplicial
if all its faces are simplices. Example: a triangle. Non-example: a square. Example: two square-based pyramids joined at the bases.
Let $P$ be a $d$-dimensional simplicial convex polytope with $f_i$ $i$-dimensional faces, $0 \leq i \leq d-1.$
The $f$-vector ($f$ for faces) of $P$ is the vector $f(P)=(f_0, \ldots, f_{d-1})$. The $h$-vector of $P$  is
defined by setting $h(P)= (h_0, \ldots, h_d)$ with
\[
h_i = \sum_{j=0}^i \binom{d-j}{d-i} (-1)^{i-j} f_{j-1} \quad (f_{-1}:=1)
\]
The $f$ and $h$-vectors determine each other.

\begin{??}\la{qpoly}When is a vector $(f_0, \ldots, f_{d-1})$ an  $f(P)$-vector for some $P$?
\end{??}

\textbf{(A reformulation of) P. McMullen's 1971 conjecture}.  (Of course, to  give $f$ is the same as  giving $h$.)
A vector $f$ is an $f(P)$-vector for some $P$ iff  1) ${h_{i} = h_{d-i}}$ and 2) there is a graded commutative  $\rat$-algebra $R=\oplus_{i\geq 0} R_i$,
with $R_0=\rat$,
generated by $R_1$, and with $\dim{R_i}=h_{i} - h_{i-1}$, $ 1 \leq i \leq \lfloor d/2 \rfloor$. In particular, ${h_0 \leq h_1\leq \ldots
\leq h_d}$.

\textbf{Associated simplicial toric variety.}
Stanley himself writes: ``we are led to suspect the existence of a smooth $d$-dimensional projective variety $X(P)$ for which
(the Betti numbers) $b_{2i}=h_i$'', for which $H^{even}(X,\rat)$ is generated by $H^2(X,\rat)$  and we can take
for $R:= H^{even}(X,\rat)/(\eta)$ (quotient by the ideal generated by hyperplane class). 

Then 1) above  would be Poincar\'e duality and 2) would be a direct consequence of hard Lefschetz on the smooth projective $X$
(unimodality).

Let me describe briefly what is going on. 

The combinatorial data of the simplicial $P$ gives rise to a simplicial toric {variety~$X(P)$}. 

Saying that $P$ is simplicial   means that $X(P)$, while possibly   singular,  has singularities  of the type ``vector space modulo a finite group''.

\textbf{Necessity of the conditions.}
It is a fact that $H^*(X(P), \rat)= H^{even}(X(P), \rat)$ and that $b_{2i} =h_i$.
The basic idea is that: faces give rise to torus orbits; torus orbits assemble into cells with the shape of
affine spaces modulo finite groups;  then $X(P)$ is a disjoint union of  such cells;  since the are  cells
automatically of  even real dimension, the cohomology
has graded bases  labelled by these cells; the only issue is to count these cells properly;   this is indeed the explanation
of the relation $f(P) \leftrightarrow h(P)$ (in the simplicial case).

Exercise~\ref{5678} tells us that  $\m{IC}_{X(P)} = \rat_{X(P)}$.
It follows that $H^*(X(P), \rat)$  satisfies Poincar\'e duality. We thus get  the necessity of 1) in McMullen's conjecture.

The necessity of 2) would follow if we knew the hard Lefschetz theorem for $H^*(X(P), \rat)$.
But we do, because we know it for the intersection cohomology groups, which, by $\m{IC} =\rat$,
are the cohomology groups.

We thus have.

\begin{tm}\la{mcbils}
{\rm (\textbf{Simplicial polytopes: iff for $f$ being an $f(P)$ vector})}
The McMullen conditions are necessary (Stanley: discussion above) and sufficient (Billera and Lee: construction).
\end{tm}

\subsection{Intersection  cohomology of the target as direct summand}\la{dr43}
\textbf{A funny situation.}
 Singular cohomology is functorial, but, in general, $f^*$ is not injective, not even if $f$ is proper\footnote{
It is injective when  $f$is  proper of algebraic manifolds (trace map)}. On the other hand,
 intersection cohomology is not functorial,
  but for proper surjective  maps, the DT exhibits the 
 intersection cohomology of the target as a direct summand of the intersection cohomology
 of the source. 

The following theorem is one of the most striking  and useful applications of the decomposition  theorem. It is usually
used, stated and proved in the context of proper birational maps. The proof
in the presence of generic large fibers is not more difficult. In fact, we  give a proof as it is also a chance to
meet and use some very useful general principles of the theory we have been talking about in these lectures.

\begin{tm}\la{icdr}{\rm (\textbf{Intersection complex as a direct summand})}
Let $f: X \to Y$ be a proper map of irreducible varieties with image $Y'$. Then $IH^*(Y')$
is a direct summand of $IH^*(X)$. More precisely, $\m{IC}_{Y'}$ is a direct summand
of $f_* \m{IC}_X.$
\end{tm}

Let us state three general and useful principles. 
Recall that for a given $(S,L)$,  the locally constant sheaf $L$ is only defined on a suitable 
open dense subvariety $S^oS_{reg}$ of $S$ and that one can shrink $S^o$, if necessary.

\begin{fact}\la{locpri}{\rm 
(\textbf{$IC$ Localization Principle})} Let  $IC_S (L) \in P(Y)$, so that $S\subseteq Y$ is closed, and let $U \subseteq Y$
be open. Then  $\m{IC}_S (L)_{|U}  =\m{IC}_{S \cap U} (L_{|S^o\cap U}).$
\end{fact}

\begin{fact}\la{nocpri}{\rm 
(\textbf{$IC$ Normalization Principle})}  Let $\nu:\hat{Y} \to Y$ be the normalization of a variety.
Then $\nu_* \m{IC}_{\hat{Y}} (L) = \m{IC}_Y (L)$ (here,  $\nu$ is finite, so that $R\nu_*=\nu_*$, derived=underived).
\end{fact}

These first two principles hold because intersection complexes with coefficients are characterized by the strengthened conditions of support
and by restricting to the locally constant sheaf on some Zariski dense open subset of the regular part, and both conditions are preserved under restriction to any open set and under  a finite birational map.  See also Exercise~\ref{staic}.

\begin{fact}\la{locprin}
{\rm (\textbf{DT Localization Principle})} A summand $\m{IC}_S(L)$ appears in the decomposition theorem
on $Y$
 iff   there is an open $U \subseteq Y$ meeting $S$ such that
 the restriction $\m{IC}_S(L)_{|U}$  
 appears  in the DT on $U$. 
 \end{fact}
 
 This  last principle is very important. It fails, for example, for the map from the Hopf surface to $\pn{1}$
 in the following sense:
 there is no decomposition theorem over $\pn{1}$ (else we would have $E_2$-degeneration of the LSS),
 but the Hopf map is locally trivial over any  open proper subset $U \subseteq \pn{1}$, so that the 
 decomposition theorem holds there (K\"unneth).  It is important because when looking for summands in the decomposition theorem, it may be easier to detect them over some Zariski open subset. For example, if $X$ is nonsingular,
 given $f :X \to Y$, there is the open subset $Y_{reg(f)} \subseteq Y$ of regular values of
 $f$. Let $\sq$ be the locally constant sheaves given by the cohomology of the fibers of $f$ over $Y_{reg(f)}.$ Deligne's theorem applies to the map over $Y_{reg(f)}$. The reader can now observe that the  DT localization principle
 allows us to deduce that  all the   $\m{IC}_Y(R^q)[-q]$ are direct summands of $Rf_* \rat_X.$
 
 The principle follows from the validity of the decomposition theorem  on $Y$ and on every $U$ and from the fact that
 the summands of the decomposition theorem  over $U$ are uniquely determined (this is left as Exercise~\ref{uniqdet}).

 In the proof of Theorem~\ref{icdr}  we shall also make use of a simple fact concerning topological coverings
 that we leave as Exercise~\ref{trcov}

\begin{proof}
\n
\textbf{\em (of  Theorem~\ref{icdr})}

\bit
\item
WLOG, we may assume that $Y':= f(X)=Y,$  i.e. that $f$ is surjective.

\item
We have $f_* \m{IC}_X
\cong \oplus_{q\geq 0} \oplus_{(S,L) \in \m{EV}_q}  \m{IC}_S (L)[-q].$

\item
By the two localization principles above, we can replace $Y$ with any  of its Zariski-dense open subsets. 

\item
We may thus assume that there are no  enriched proper subvarieties  in the~DT:
\[f_* \m{IC}_X 
\cong \oplus_{q\geq 0} \m{IC}_Y (L_q)[-q].\]

\item
By constructibility,  and by further shrinking if necessary, may also assume  that $\m{IC}_Y (L_q)= L_q$ is locally constant and we get
\[Rf_* \m{IC}_X \cong  \bigoplus_{q\geq 0} R^qf_* \m{IC}^q [-q] =
 \bigoplus_{q\geq 0} L_q [-q].\]

\item
We may also assume that $\m{IC}_Y = \rat_Y.$

\item WLOG, we may assume that $X$ is normal. In fact, take the normalization $\nu: X' \to X$;
by the $IC$ Normalization Principle, we have $R\nu_* \m{IC}_{X'} = \m{IC}_X;$ on the other hand, we have $R(f\circ \nu)_*=
Rf_* \circ R\nu_*$  ($\nu$ is finite, so $R\nu_* =R^0\nu_*$; but we do not need this here).

\item
FACT: regardless of normality, there is always  a natural map
$\rat_X \to \m{IC}_X$ in place.  Since $X$ is normal, this map   induces
$\rat_X \cong \m{H}^0 (\m{IC}_X),$ i.e. we have a distinguished triangle
$\rat_X \to \m{IC}_X \to \tu{1} \m{IC}_X \stackrel{[1]}\to .$ We push it forward and deduce
\[
R^0f_* \rat_X = R^0f_* \m{IC}_X.
\]

\item
We are thus reduced to showing that, up to further-shrinkage,
$\rat_Y$ is a direct summand of $L_0=R^0f_* \m{IC}_X=R^0f_* \rat_X.$

\item
Stein factorize $f:=h_{finite} \circ g_{connected \, fibers}= : X \stackrel{g \,conn.fib.}\lorw Y' \stackrel{h\, fin.}\lorw Y.$

\item
Because of connected fibers, we have $R^0g_* \rat_X= \rat_Y.$
We are thus reduced to showing that, after shrinking,   $\rat_Y$ is a direct summand
of  $R^0h_* \rat_{Y'}.$

\item
Since $h$ is finite,  by shrinking the target if necessary, $h$ becomes a covering map.
\item
The claim now follows by  a standard trace argument (Exercise~\ref{trcov}).\qedhere
\eit
\end{proof}

\subsection{Pure Hodge structure in intersection cohomology}\la{pfgt}

The theory of mixed Hodge modules of M. Saito \ci{samhm} endows the 
intersection cohomology groups of complex varieties  with a mixed Hodge structure. 

Let us use the intersection complex as a direct summand Theorem~\ref{icdr} to endow
the intersection cohomology of a complete surface with a pure Hodge structure.
This is merely to illustrate the method, which works for any algebraic variety \cites{htadt,decII,split}. This special case is 
simple because the resolution is semismall, yet illuminating because its simplified set-up allows to focus on the main 
ideas  without  distractions.

Let $Y$ be a complete surface.  We are interested in $IH^*(Y,\rat)$, so that we may assume
that $Y$ is normal, for the intersection cohomology groups
do not change under normalization. In particular, $Y$ has isolated singularities. Let $S$ be the finite set of singular points. Pick a resolution of the singularities $f:X \to Y$
that leave $Y_{reg}$ untouched. The decomposition theorem has the form
$Rf_* \rat_X [2] = IC_Y \oplus \oplus_{b} V^b_S[-b]$, where $V^b_S$ is a skyscraper sheaf at $S$.
Since all fibers have dimension $\leq 1,$ we have $R^{q \geq 3}f_* \rat_X =0$.
It follows that $V^{b\geq 1}_S =0$. By the symmetries of Verdier duality, we have that
$V^{b\leq -1}_S=0$ and we have $Rf_*\rat_X[2]  \cong IC_Y  \oplus
V^0_S$. This is perverse and self-dual. The associated pairing yields the  usual intersection pairing
on $H^*(X,\rat)= IH^*(Y,\rat) \oplus V_S$. The two summands are orthogonal for this pairing
(there are no maps $IC_Y \to V^0_S$!). The l.h.s. is a pure Hodge structure. The pairing is a map of pure Hodge structures.
In order to conclude that $IH^*(Y,\rat)$ is a pure Hodge substructure (our goal!), it is enough to show that
$V^0_S \subseteq H^2(X,\rat)$ is a pure Hodge substructure. By the support conditions for $IC$, ${\m{H}^0(IC_Y)=0}$. It follows that $V^0_S=R^2f_* \rat_X= H^2(f^{-1}(S),\rat)$, which is generated 
by the fundamental classes of the curve fibers, which are of $(p,q)$-type $(1,1)$, i.e. they form a pure Hodge structure
of weight two.

\subsection{Exercises for Lecture 4}\la{Exelz3}

\begin{exe}\la{de3bg}{\rm (\textbf{Failure of local Poincar\'e duality;  $\rat_Y$ not a direct summand})
 Let $Y$ be the affine cone over a nonsingular embedded projective curve
of genus $g \geq 1.$  Use the defining property $\omega_Y$ to show that
$\omega_Y \neq \rat_Y [4]$ so that the usual local Poincar\'e duality fails. 
Take the usual resolutions $f: X \to Y$. Use the fundamental relation $Rf_* (C^\vee) =
(Rf_*C)^\vee$ and deduce, by using the failure of  Poincar\'e
duality in neighborhoods of the vertex, that $\rat_Y$ is not a direct summand of $Rf_* \rat_X.$}
\end{exe}

\begin{exe}\la{5tgb}
{\rm (\textbf{Goresky-MacPherson's estimate})
 Let $f: X\to Y$ be proper with $X$ nonsingular (or at least with $\m{IC}_X=\rat_X$). Assume that $S \subseteq Y$
 appears in the decomposition theorem for $Rf_*\rat_X[\dim{X}].$ Show that \[\dim{X} - \dim{S} \leq 2 \dim f^{-1}(s), \quad \forall s \in S.\]
Observe that equality implies that  $S$ appears only in perversity zero (i.e. with shift $b=0$ only).
Deduce from this that, for example, in the decomposition theorem for the mall resolution of the affine cone over $\pn{1}\times
\pn{1} \subseteq \pn{3}$, the vertex does not contribute direct summands. What happens if the map $f$ is of pure
relative dimension $1$? (Relate the answer to the number of irreducible components in the fibers and consider
what kind of very special property the intersection complexes appearing in the decomposition theorem 
should enjoy). What happens if, in addition to 
having $f$ of relative dimension one, $Y$ is also nonsingular, or at least has $\m{IC}_Y=\rat_Y$?}
\end{exe}

\begin{exe}\la{pld}{\rm(\textbf{Primitive Lefschetz decomposition}) Familiarize yourself with the primitive Lefschetz
decomposition (PLD) associated with the hard Lefschetz theorem for compact K\"ahler manifolds.
Observe that the same proof  of the  PLD holds if you start with a bounded  graded object $H^*$
in an abelian category with $H^b=0$ for $|b|\gg 0$ and  endowed with a degree two operator $\eta$ satisfying
$\eta^{b}: H^{-b} \cong H^b$ for every $b \geq 0.$ Deduce the appropriate PLD for the graded object $\pcs^{*}$
arising from the relative hard Lefschetz theorem~\ref{rhl}. State an appropriate version of the 
unimodality  of the Betti numbers of  K\"ahler manifolds in the context of the graded object $\pcs^* \in P(Y).$}
\end{exe}

\begin{exe}\la{hl} {\rm (\textbf{Proofs of hard  (vache (!), in French) Lefschetz })
Let $i: Y \to X$ be a nonsingular hyperplane section of a nonsingular projective manifold. Use Poincar\'e duality and the slogan 
``cup product in cohomology = transverse intersection in homology, to show that we have commutative diagrams
for every $r\geq 1$ (for $r=1$ we get a triangle!)
\beq\la{rtm}
\xymatrix{
H^{\dim{X} -r}(X,\rat) \ar[d]^{i^*}_{\mbox{restriction}} \ar[r]^{\eta^r} & H^{\dim{X} +r}(X,\rat)  \\
H^{\dim{Y} -(r -1)}(Y,\rat)  \ar[r]^{\eta_{|Y}^{r-1}} & H^{\dim{Y} +(r-1)}(X,\rat)  \ar[u]^{i_!}_{\mbox{Gysin}}
}
\eeq
Assume the Hard Lefschetz for $Y$ (induction). Use the Lefschetz hyperplane theorem and deduce  the hard Lefschetz  for $X$, but only for $r\geq 2$ (for $r=0$ it is trivial). For $r=1$, we have the commutative triangle with $i^*$ injective and $i_!$ surjective
\beq\la{rtmb}
\xymatrix{
H^{\dim{X} -1}(X,\rat) \ar[dr]^{i^*}_{\mbox{restriction}} \ar[rr]^{\eta} && H^{\dim{X} +1}(X,\rat)  \\
& H^{\dim{Y} }(Y,\rat)   \ar[ur]^{i_!}_{\mbox{Gysin}}.&
}
\eeq
Hard Lefschetz boils down to the statement that $\im\,{i^*} \cap \ke
\, i_! =\{0\}.$ 
Show that hard Lefschetz is equivalent to the statement: $(*)$  the non-degenerate intersection form
on $H^{\dim{Y}}(Y,\rat)$ stays non-degenerate when restricted to $i^*H^{\dim{X}-1}(X,\rat)$.
At this point, we have two options. Option 1:  use the Hodge-Riemann bilinear relations for $Y$:
a class $i^*a \neq 0$ in the intersection would be primitive and the same would be true for its  $(p,q)$-components;
argue that we may assume $a$ to be of type $(p,q)$;
the Hodge-Riemann relations would then imply $0 \neq  \int_Y (i^* a)^2= \int_X \eta \wedge a^2=0,$
a contradiction. 
 Option 2: put $Y$ in a pencil  $\w{X} \to \pn{1}$ with smooth total space (blow up $X$); let $\Sigma$ be the set of critical values of $f$; use the global invariant cycle
theorem  and the Deligne semisimplicity to show that the irreducible $\pi_1 (\pn{1} \setminus \Sigma)$-module $H^{\dim{Y}}(Y,\rat)$
has $i^* H^{\dim Y}(X,\rat)$ as its module of invariants. Conclude by first proving and then by   using the following lemma (\ci{weil2}, p.218):
let $V$ a completely reducible  linear representation of a group $\pi$, endowed with a $\pi$-invariant and  non-degenerate bilinear form $\Phi$; then the restriction of $\Phi$ to the invariants $V^\pi$ is non-degenerate.
}
\end{exe}

\begin{exe}\la{moresymme}
{\rm
(\textbf{Four symmetries})
Let $f:X\to Y$ be a projective map of projective varieties  with $X$  nonsingular and  $Y$ of dimension $4$.
Assume that the direct image perverse cohomology sheaves $\pcs^b$ of $Rf_*\rat_X[\dim{X}]$ live in the interval $[-2,2]$
and that  the  enriched varieties appearing in the decomposition theorem  are supported at  a point  $S^0$, 
at curve $S^1$ and at a  surfaces $S^2$
on $Y$. Denote the $b$-th direct summands of  $H^*(X,\rat)$ stemming from (\ref{mki9})  by $H^*_b(X,\rat)$.
Each of these groups further decomposes into direct summands labelled by the ${S_k: H^*_b(X,\rat)=\oplus_{k} H^*_{b,S^k}(X,\rat)}$.  List  the four kind of symmetries
among the various $H^*_{b,S^k}(X,\rat)$: Verdier Duality  and Relative Hard Lefschetz
for $b$ and $-b$; Verdier duality and Hard Lefschetz in intersection cohomology for the  for the same $b$.
(See \ci{decmightam}, \S2.4.)
}
\end{exe}

\begin{exe}\la{okn}
{\rm
(\textbf{Perverse cohomology sheaves and the decomposition theorem})
Show that if $P\in P(Y)$, then $P=\pc{0}{P}.$ Show that if $C= \bigoplus P_b[-b]$ with ${P_b\in P(Y)}$,
then $\pc{b}{C} = P_b.$
Let $j: \comp^n\setminus \{0\} \to \comp^n$ be the natural open embedding.
Compute $\pc{i}{Rj_! \rat}$ and $\pc{i}{Rj_* \rat}$. Let $j:U\to X$ be an open embedding and let $P \in P(U)$.
Deduce that in the decomposition theorem for $C:= Rf_* IC_{S}(L)$, we have
that $\pc{b}{C}=\oplus_{EV_b} IC_S(L)$, whereas in the one for  $K:= Rf_* \m{IC}_{S}(L)$ we have a less simple expression
(involving the dimensions of the varieties $S$).
}
\end{exe}

\begin{exe}\la{uniqdet} {\rm (\textbf{The summands in the DT are uniquely determined})}
{\rm Prove that the direct summands in the DT are uniquely determined. How non-unique is the isomorphism in the statement
of the decomposition theorem?}
\end{exe}

\begin{exe}\la{noesupp} {\rm (\textbf{No extra supports})}
{\rm Let $f: X \to Y$ be a proper map with $X$ a nonsingular surface and $Y$ a curve. Use the proof of Theorem~\ref{gt6} to show that
if all fibers are irreducible, then the enriched varieties appearing in the decomposition theorem for $Rf_*\rat_X$
are supported on $Y.$ }
\end{exe}

\begin{exe}\la{5678}{\rm (\textbf{$\m{IC}=\rat$ for $\comp^n/G$})
Let  $G$ be a finite group acting on a complex vector space $V$ and let $f: X:=V\to Y:=V/G$ be the resulting
finite quotient map. Show that the natural map $\rat_Y \to R^0f_* \rat_X$ splits. Deduce that
$\rat_Y[\dim{Y}]$ is Verdier self-dual. Deduce that $\rat_Y[\dim{Y}]$ is perverse.
Show that moreover, it satisfies the conditions of support that characterize the intersection complex~${IC}_Y.$
}
\end{exe}

\begin{exe}\la{staic}{\rm 
(\textbf{$IC$  and finite maps})
Prove that the direct image of an intersection complex under a finite map
is an intersection complex.}
\end{exe}

\begin{exe}\la{trcov} {\rm  (\textbf{Coverings traces})} {\rm Let $f: X \to Y$ be a proper submersion of fiber dimension zero (finite topological covering). Prove (without using the  fancy semisimplicity results seen above) that $L:=f_*\rat_X$ is a semisimple locally constant sheaf admitting $\rat_Y$ as a direct summand (trace map).
If the covering is normal (a.k.a. Galois), use the language of representations of the fundamental group to reach the same conclusion.}
\end{exe}

\section{Lecture 5: the perverse filtration}\la{tpfg}

\textbf{Summary of Lecture 5.}  {\em The classical ``topologists'' Leray spectral sequence for fiber bundles. Grothendieck spectral sequence, Leray as a special case. Verdier's spectral objects. Geometric description of the
perverse Leray filtration. Relation  to the topologists' point of view. Hodge-theoretic applications.
The P=W  theorem and conjecture in non abelian Hodge theory. A sample perversity calculation. A motivic question on the projectors that can be associated with the decomposition theorem.}

\subsection{The perverse spectral sequence and the perverse filtration}\la{rmtyo}

It is important to keep in mind that given $f:X\to Y$,  the Leray filtration $L$  on $H^*(X,\rat)$ is defined  a priori, independently
of the Leray spectral sequence. The latter is  machinery that tells you something about the
graded pieces $Gr^L_i H^d(X,\rat)$ of the Leray filtration, i.e. $Gr^L_i H^d(X,\rat)  =E_\infty^{d-i,i}.$

\textbf{The Leray spectral sequence for a fiber bundle.} You can consult Spanier's or Hatcher's algebraic topology textbooks.
Let $p: E \to B$ be a topological fiber bundle with fiber $F$.
Assume you are given a cell complex structure $B_\bullet$ on $B$: in short, we have
the $p$-th skeleta $B_p$, $B_p\setminus B_{p-1}$ is a disjoint union of $p$-cells, $H^r(B_p,B_{p-1})
=\w{H}^r(\mbox{Bouquet of $p$ spheres})\cong \delta_{rp}\rat^{\#}$ (we call this the cellularity condition), etc.
The cohomology of the complex $H^p(B_p, B_{p-1})$ with differentials given
by consideration of the coboundary operators in the long exact sequence of the triples
$(B_p, B_{p-1}, B_{p-2})$, computes $H^*(B, \rat)$. In fact, there is a spectral sequence
$E_1^{p,q} = H^{p+q}(B_s,B_{s-1})$, but the cellularity condition reduces the spectral sequence to a complex.
The same kind of spectral sequence for the pre-images $E_\bullet$ of $B_\bullet$ reads
$E_1^{p,q} = H^{p+q}(E_p, E_{p-1})$ and it does not reduce to a complex. The  bundle structure and the cellularity condition
tell us that $E_1^{p,q} = H^p(B_p,B_{p-1})\otimes H^q(F).$ One then argues that $E_2^{pq}=H^p(B,\sq)$.
This is probably close in spirit to the original way of viewing the Leray spectral sequence for a fiber bundle.
The increasing Leray filtration on the cohomology of the total space  is given by the kernel of the restriction maps
to pre images of the skeleta $\ke H^?(E,\rat) \to  H^?(E_{??},\rat)$. Let us not worry about indexing schemes.

\textbf{Grothendieck's Leray spectral sequence.} You can consult Grothendieck ``Tohoku'' paper.
Grothendieck gave a sheaf-theoretic approach to this picture: start with a complex of sheaves $C$ on $Y$; take a Cartan-Eilenberg resolution
for $C$, i.e. an injective resolution  $C \to I$ that ``is'' also an injective resolution for the truncated complexes $\td{i} C$ and for the cohomology sheaves $\m{H}^i(C);$ the complex of global sections $\Gamma (Y,I)$   is filtered; the Grothendieck
spectral sequence is the spectral sequence for this filtered complex, and it abuts to the  standard (Grothendieck) filtration
given by $\im\,H^*(Y, \td{i}C) \subseteq H^*(Y,C).$ Given $f: X \to Y$ and $C \in D(X)$, the (Grothendieck-)Leray
spectral sequence is the Grothendieck spectral sequence for $Rf_* C$, and the (Grothendieck)-Leray filtration 
is the standard filtration for~$Rf_*C$. 

\textbf{Perverse and perverse Leray spectral sequences and filtrations.}
The main input leading to the machinery above, is the use of injective resolutions
together with the system of standard truncation maps.
If we replace the standard truncation with the perverse truncation maps, we obtain the perverse and perverse Leray spectral sequences
and the perverse and perverse Leray filtrations. At the end, given $C \in D(Y)$, the perverse filtration
is given by setting $\ms{P}_b H^*(Y,C) := \im \,H^*(Y, \ptd{b} C) \subseteq H^*(Y, C)$. Similarly for the perverse Leray filtration
\[\ms{P}_b H^*(X,Rf_* K) := \im \,H^*(Y, \ptd{b} Rf_*K) \subseteq H^*(Y, Rf_*K)= H^*(X,K).\]
In fact, every t-structure on $D(Y)$, and there are many different ones!, gives rise to the same kind of picture
outlined above. The Grothendieck and Leray filtrations correspond to the standard t-structure
on~$D(Y)$.

\textbf{Verdier's spectral objects.}
There is at least another convenient way to view these mechanisms, one that, once you are given a cohomological functor,
 avoids injective resolutions, namely Verdier's spectral objects \ci{shockwave}: 
the input is a cohomological functor on a t-category; the output is a spectral sequence abutting to the filtration
defined in cohomology by the t-truncations. Even if not logically necessary in our context, this is  a useful tool when there are not
enough  injectives (e.g. the category of perverse sheaves
does not have enough injectives), and it is a, yet another!,  cool way to look at spectral sequences.

\textbf{Why bother with the perverse Leray filtration?} Because in the context of the decomposition 
theorem \[C:= Rf_* IC_X(M) \cong \oplus_b \oplus_{EV_b} IC_S(L) [-b]= 
\oplus_b \pc{b}{C} [-b],\] the perverse Leray spectral sequence is $E_2$-degenerate,  
the graded pieces of the perverse Leray filtration 
are the cohomology groups $H^{*}(Y, \pc{b}{C})$ and  the cohomological decomposition theorem
gives splittings of the perverse Leray filtration. In particular, every cohomology class
in $IH^*(X,M)$  has $b$-components, which split further into $EV_b$-components.
In particular, if you know that the perverse filtration has some property,
e.g. it carries a Hodge structure, then the graded $b$-pieces inherit  such a structure as well, and so will
the individual $EV_b$-pieces. See Corollary~\ref{bnmy} and Remark~\ref{stream}

\textbf{Skleta in algebraic geometry?}
Let us go back to Leray for fiber bundles. In that topological context, it is natural to work with a cell complex structure.
We can do that with smooth projective maps $f: X\to Y$ in complex algebraic geometry (varieties can be triangulated), but the skeleta 
will not be algebraic subvarieties. It is hard to predict the properties of the Leray filtration if it is described 
as a kernel of a restriction map to some closed subspace that is not a subvariety. We may say that in the context of topological fiber bundles,
the Leray filtration is described geometrically (using the geometry at hand) via the kernels of the  pull-back maps
to pre-images of skeleta.
In the context of maps of complex algebraic varieties, we are interested in a geometric description of the perverse and perverse Leray filtrations, but the skeleta of a cell-decomposition do not seem to be immediately helpful.

\subsection{Geometric description of the perverse filtration}\la{rmtyop}
\begin{??}\la{qonpf} 
Can we describe the perverse filtration $\ms{P}_b H^*(Y,C)$ geometrically?
\end{??}

Let us approach this problem in the special case that is reminiscent to  the decomposition theorem, i.e. let us assume
that $C=\oplus_b P_b [-b] \in D(Y)$ with $P_b \in P(Y)$. In this case, we have that
the perverse filtration comes to us already canonically  split: 
\beq\la{pspl}
\ms{P}_b H^*(Y,C) = \im \left( \oplus_{b'\leq b} H^{*-b}(Y, P_b)\right) \subseteq H^*(Y,C).
\eeq
It is a good place to remark that the decomposition theorem asserts the existence
of a direct sum decomposition, not that one can find a natural one. In general, there is no 
such thing. A relatively ample line bundle provides you with the possibility of choosing
some distinguished splittings; see \ci{shockwave} and \ci{split}.

Exercise~\ref{ght} proves half of the following fact: let $P$ be a perverse sheaf on a quasi projective variety
$Y$ and let $Y_\bullet$ be a  general flag of  linear sections  of $Y$ for some embedding in projective space; 
here $Y_k \subseteq  Y$ has codimension $k$ in $Y;$ then $P_{|Y_k}[-k]  \in P(Y_k).$

Exercise~\ref{ght2}  first asks you to apply repeatedly  the Lefschetz hyperplane theorem for perverse sheaves 
to the elements of the flag $Y_\bullet$ to show that the restriction maps
$H^*(Y,P) \to H^*(Y_k, P_{|Y_k})$ are  injective
for every $* \leq -k.$ Next, it asks 
 you to specialize the situation to the case when $Y$ is affine, to use Artin vanishing theorem and deduce that, for $Y$ affine, the
 restriction maps
$H^*(Y,P) \to H^*(Y_k, P_{|Y_k})$   are zero for $* >-k.$

We   conclude that, when $C= \oplus_b P_b [-b] \in D(Y)$, with $P_b \in P(Y)$ and $Y$ is affine, we have 
a geometric description of the perverse filtration
\beq\la{sde}
\ke{(H^*(Y,C) \to H^*(Y_k, C_{|Y_k}))}=   \ms{P}_{*+k-1} H^*(Y,C).
\eeq
In other words, for cohomological purposes, we may consider the $Y_\bullet$ as the skeleta
of a ``cell'' decomposition; the term cells now refer to the fact that the relative cohomology groups
$H^*(Y_{k},Y_{k+1},P)$ are non zero in at most one cohomological degree, which is reminiscent of the analogous fact
for cell complexes (vanishing for bouquet of spheres).

By renumbering ~\ref{sde}, we get
\beq\la{gdpvb}
\ms{P}_bH^*(Y,C)= \ke{(H^*(Y,C) \to H^*(Y_{b-*+1}, C_{|Y_{b-*+1}}))}.
\eeq

What if $C\in D(Y)$ is not split and $Y$ is not affine? 

If $Y$ is affine, then exact same description, but with a different proof, remains valid  for every $C \in D(Y)$.

If $Y$ is quasi projective, then we can use the  Jouanolou  trick  (Exercise~\ref{joau}), to reduce to the case to the
affine situation; see  \ci{decII}. The use of this trick is not necessary and one can work directly on $Y$, 
but has to use a general pair of flags coming from a suitable embedding in projective space.

Let us state the end result for the  perverse Leray filtration  in the special, but key case of a map to an affine
variety. Note that the map needs not to be proper and that it applies to every complex, not just one whose direct image splits
as above. For a discussion of the case of the standard Leray filtration, see
\ci{deabday}.

\begin{tm}\la{tmgdpf}{\rm (\textbf{Geometric description of the perverse Leray filtration} \cites{decmigso3,decII})}
Let $f:X \to Y$ be a map of varieties with $Y$ affine and let $K \in D(X).$ Then there is a flag
$Y_\bullet \subseteq Y$, with pre-image flag $X_\bullet \subseteq X$, such that
\[
\ms{P}_b H^*(X, K) =  \ke\left\{(H^*(X,K) \to H^*(X_{b-*+1}, K_{|X_{b-*+1}})\right\}.
\]
\end{tm}

\subsection{Hodge-theoretic consequences}\la{htco}
Here is a corollary that exemplifies the utility of having a geometric description
of the perverse filtration as the  kernels of restriction maps.

\begin{cor}\la{bnmy} {\rm (\textbf{Perverse Leray and MHS})}
 Let $f: X \to Y$ be a map of varieties. Then
the subspaces  $\ms{P}_b \subseteq H^*(X,\rat)$ of  the perverse Leray filtration  on the cohomology
of the domain are mixed Hodge substructures.  In particular, the graded pieces
carry a natural mixed Hodge structure.
\end{cor}

\begin{rmk}\la{stream}{\rm (\textbf{Streamlined Hodge-theoretic proof of the decomposition theorem})
Corollary~\ref{bnmy} can be used as the basis for 
a streamlined  Hodge-theoretic proof of
the decomposition  and allied results for the push-forward of the intersection complex of a variety. This would shorten considerably the proofs in \ci{decmightam}.}
\end{rmk}

\subsection{Character variety and Higgs moduli: \texorpdfstring{$P=W$}{P=W}}\la{p=w}

There is a version of the story outlined in what follows  for any complex reductive group; one can even consider higher dimensional 
projective manifolds instead of just curves.  However,  we stick with  $GL(2,\comp)$ and curves.

Fix a smooth projective curve $X$ of genus $g\geq 2$ a point $x\in X$.

\textbf{Character variety $M_B$.} 
Let $M_B$ be the moduli space of irreducible representations of $\pi_1(X\setminus x) \to GL(2,\comp)$ subject to the condition
that a small loop circuiting $x$ maps to $-Id.$ This is a nonsingular affine variety of dimension
$2a:=8g-6.$ 

\textbf{Higgs moduli space and Hitchin map.}
Let $M_D$ be the moduli space of stable rank two and degree one  Higgs bundles $(E,\phi)$ on $X$, 
where $\phi$ is a one form with coefficients in $\text{End}(E).$ This is a nonsingular quasi projective variety
of the same dimension $2a$. It is quasi projective, neither affine, nor quasi projective: it carries the 
projective Hitchin map $h: M_D \to A\cong\comp^{a}$, $(E,\phi) \mapsto (\mbox{trace}(\phi), \mbox{det}(\phi))$
(sections of $T^*X$ and of its tensor square), and the general fiber is an abelian variety
of dimension $a$. 

\textbf{Non abelian Hodge theorem.}
Part of the  non abelian Hodge theorem states that these  two varieties are naturally diffeomorphic. They are not
biholomorphic. 
There is also a third moduli space in the picture, related to flat connections; but we stick to our limited set-up above.

\textbf{Mixed Hodge structures on $H^*(M_D)$ and $H^*(M_B)$.}
The mixed Hodge theory of both sides is relatively well-understood: the one for $M_D$ is pure (Exercise~\ref{mdpure}),
 so that the weight filtration 
$W_D$ on $H^*(M_D,\rat)$   is  the filtration by cohomological degree: $0= W_{D,d-1} H^d \subseteq
W_{D,d}H^d=H^d$, or $Gr^{W_D}_d H^d=H^d$. The mixed Hodge structure
 on $H^*(M_B)$ is  more interesting: the odd graded pieces $Gr^{W_B}_{\text{odd}} H^d=0$, while 
 the even ones $Gr^{W_B}_{2b}H^d$ are of pure type $(b,b)$ (Hodge-Tate) and the non-purity, here,
 refers to the fact that we do have degrees $d$ for which  $H^d=\oplus_{b}Gr^{W_B}_{2b} H^d$ with 
 more than one non trivial summand; the non trivial graded pieces live
in the interval $[0, 4a]$.

\textbf{$W_B$ and $W_D$ do not match.}
By what above, it is clear that the weight filtrations $W_B$ (mixedness) and $W_D$ (purity)  do not correspond under the diffeomorphism $M_B \cong M_D.$

\textbf{The curious hard Lefschetz phenomenon on $H^*(M_B, \rat)$.}
There is a distinguished cohomology class $\alpha\in H^2(M,\rat)$ which is linked to  a curious  phenomenon concerning the $M_B$-side, i.e. 
there is a sort of hard Lefschetz statement of the form: 
\beq\la{chl}
\alpha^b \cup -: Gr_{2a- 2b}^{{W_B}} H^*({M_B}) \stackrel{\cong}\lorw 
Gr_{2a +2b}^{{W_B}}     H^{*+2b}({M_B}) .
\eeq
 It is called the curious hard Lefschetz because it looks like a hard Lefschetz-kind of statement, but   remember that  $M_B$ is affine
and that  $\alpha_B$ is  a  $(2,2)$ class!

\begin{??}\la{gt5}
Via the non abelian Hodge theorem isomorphism \[H^*(M_B,\rat) \cong H^*(M_D,\rat), \]what corresponds
 to $W_B$ together with its curious hard Lefschetz, on the $H^*(M_D,\rat)$-side?
\end{??}

To answer this question,
 we first normalize the perverse Leray filtration $\ms{P}$  on $H^*(M_D,\rat)$ for the Hitchin map
 so that its graded pieces are trivial outside the interval
$[0, 2a]$. We denote the result by $\ms{P}_D.$ This means that instead of working with $Rf_* \rat_{M_D}$, we work with $Rf_*\rat_{M_D}[a]$ (Exercise~\ref{normperv}).  

Recall that the analogous interval for the weight filtration $W_B$ is $[0,4a]$ and that the
odd graded pieces are trivial. This makes the following answer to Question~\ref{gt5} ``numerically'' plausible:
intervals match by halving, in view of $[0,2a]$ and $[0,4a]$ and $Gr^{W_B}_{\text{odd}}=0$.
On the $CHL/RHL$-side, the answer is also made plausible by the fact that the class $\alpha$ is
ample on the fibers of the Hitchin map, so that indeed it gives rise to a RHL.

\begin{tm}\la{p=wtm} {\rm (\textbf{P=W} \ci{p=w})}
For $G=GL/SL/PGL(2,\comp)$, via the non abelian Hodge theorem  $M_B\cong M_D$, we have:
\[
W_{B,2b} \longleftrightarrow  \ms{P}_{D,b} \quad \forall b, \qquad \mbox{CHL} \longleftrightarrow  \mbox{RHL}.
\]
\end{tm}

Of course, even if numerically plausible, the fact that the subspaces of the filtrations match and that
CHL turns into RHL seems striking to some of us.

The proof of Theorem~\ref{p=wtm} makes  an essential use of the geometric description  of the perverse filtration for the Hitchin map
based on Theorem~\ref{tmgdpf}: generators and relations  for the cohomology ring $H^*(M_B,\rat)$ 
are known (Hausel-Thaddeus, Hausel-Rodriguez Villegas); the generators are of pure type
$(p,p)$ (for various values of $p$) hence live in $W_{B,2p}$;   every cohomology class is a sum of  monomials in these generators;
such monomials have type which is the sum of the types of the factors; their level in $W_B$ is the sum
of the levels of the factors; the proof then  hinges on the verification that all monomials of level $2b$ in $W_B$
live in $\ms{P}_{D,b}$; in turn this follows by verifying that the generators have this property and, critically,  that
$\ms{P}_D$ is multiplicative (i.e. the level of a cup product is not more than the sum of the levels of the factors).

In the $GL_2/SL_2/PGL_2$-case, it is not hard to verify that the generators have the required property.
 The heart of the proof of the P=W Theorem~\ref{p=wtm}  in \ci{p=w} consists of showing that
the perverse Leray filtration for the Hitchin map is multiplicative with respect to the cup product.
This is automatic for the Leray filtration of any map, but fails in general for the perverse Leray filtration
(Exercise~\ref{pfnmu}). In our case,  the Leray filtration differs from the perverse Leray filtration.

Let us illustrate the use of Theorem~\ref{tmgdpf}  with a calculation whose result tells us that a certain generator, let us call it $\beta
\in H^4(M,\rat),$ 
of type $(2,2)$ in $W_{B,4} H^4(M_B,\rat)$, in fact  lies in $\ms{P}_{D,2} H^4(M_D,\rat)$. See \ci{p=w}, \S3.1.

By keeping in mind the  normalization  above of the perverse filtration,
 the geometric description of the perverse filtration  Theorem~\ref{tmgdpf} requires us to verify that
 the class $\beta$ vanishes over the pre-image of a generic affine line  in $A\cong \comp^a$ (end of Exercise~\ref{normperv}).
 
 The class $\beta \in H^4(M_D,\rat)$ is known to be  a multiple of  the second Chern class of the tangent bundle of $M_D.$
 Since the generic fiber is an abelian variety, it is clear that $\beta$ vanishes over the pre-image of a generic point.
 More is true: every linear function on $A$ gives rise to a Hamiltonian vector field tangent to the fibers of the Hitchin map;
 since the tangent bundle of $M_{D,reg}$ (pre-image of regular values  $A_{reg}$ of $h$) is an extension
 of the pull-back of the (trivial) tangent bundle of $A_{reg}$ by the relative tangent bundle (also trivialized by the Hamiltonian vector
 fields above), we have that in fact $\beta$ is trivial on $M_{D,reg}$.

 Ng\^{o}'s striking  support theorem \ci{ngo}  tells us that there is a Zariski dense  open set $A^{ell}\subseteq A$ with closed complement
 of codimension $>g-2$ such that
 the decomposition theorem over $A^{ell}$ is of the form $Rh^{ell}_* \rat \cong \oplus_{q\geq 0} \m{IC}_{A^{ell}}(R^q)[-q]$
 (in \ci{p=w}, we reach this conclusion directly  and complement it by showing that
 these  intersection complexes are in fact sheaves on $A^{ell}$).
 A generic line will avoid the small closed complement (at least if $g\geq 3$; $g=2$ can be dealt with separately),
 where $R^q$ are the locally constant direct image sheaves over the regular part.

 Pick a generic line $\Lambda$ and observe that the decomposition theorem for the Hitchin map restricted over the line
 reads:
 \[Rh^{\Lambda}_* \rat \cong \bigoplus_{q\geq 0} \m{IC}_{\Lambda}(R^q)[-q];\]this is because
 the restriction of an (un-shifted) intersection complex $\m{IC}$ to a general linear section
 is an (un-shifted)  intersection complex.
 
 Let $j:\Lambda_{reg}\to \Lambda$ be the open immersion of the set of regular values of $h^{\Lambda}.$
 Since $\Lambda$ is a nonsingular curve, we know that 
 \[Rh^{\Lambda}_* \rat \cong \bigoplus_{q\geq 0} j_*R^q[-q];\]
  this is because intersection complexes
 on nonsingular curves
 are  obtained via the ordinary sheaf-theoretic push-forward (Fact~\ref{509}). 
 
 Note that our perverse sheaves are now just sheaves (up to shift).
 It follows that  the perverse spectral sequence for $h^\Lambda$ is just the ordinary Leray spectral sequence and the same holds
 for  map $h^{\Lambda_{reg}}$ over the set of regular values of $h^\Lambda$.

 By the functoriality of the Leray spectral sequence and by Artin vanishing on the
 affine curves $\Lambda$ and $\Lambda_{reg}$ ($H^{>1}=0$!), we have a commutative diagram  
 \beq\la{gt54}
 \begin{gathered} \xymatrix{
 0 \ar[r]   & H^1(\Lambda, j_* R^3) \ar[r] \ar[d]  &   H^4 (M_\Lambda) \ar[r]  \ar[d] &   H^0 (A, j_* R^4) \ar[r] \ar[d]^= &  0\\
 0 \ar[r] & H^1(\Lambda_{reg}, R^3) \ar[r] &  H^4(M_{\Lambda_{reg}}) \ar[r] & H^0 (R^4) \ar[r] & 0,
 }
  \end{gathered}\eeq
of short
 exact sequences (the edge sequences for the Leray spectral sequences for the maps $h$)
 where the first vertical map is injective (edge sequence for the Leray spectral sequence for the map $j$), and the third is an
 isomorphism (definition of direct image sheaf).
 
 A simple diagram chase, tells us that the vertical restriction map in the middle of (\ref{gt54})  is also injective.
 
 On the other hand, the class $\beta_{|M_\Lambda} \mapsto \beta_{|M_{\Lambda_{reg}}}=0$ by what seen earlier
 ($\beta$~restricts to zero over $A_{reg}$, hence over $\Lambda_{reg}$).
 
 By the injectivity statement  above, we see that $\beta$ vanishes over the generic line
 and  we deduce that $\beta \in \ms{P}_{D,2}H^4(M_D,\rat)$, as predicted by $P=W$.

\begin{??}\la{alln}
We can formulate $P=W$ for every complex reductive group. Does it hold, at least for $GL_n$? 
\end{??}

There are indications that this should be ok for $GL_{n}$, $n$ small.

\subsection{Let us conclude with a motivic  question}\la{rty123}

Let $f: X \to Y$ be a projective map of projective varieties with $X$ nonsingular.
By the decomposition decomposition theorem, there is an isomorphism
\beq\la{vfgtr}
\phi: H^*(X,\rat) \cong \oplus_{q,\m{EV}_q} IH^{*-q}(S,L).
\eeq
This implies that for each $(S,L)$ in $\m{EV}_q$ we obtain a projector (map that squares to itself)
on $H^*(X, \rat)$ with image $\phi (IH^{*-q}(S,L))$.  We view this projector
as a cohomology class $\pi_\phi:= H^{2\dim{X}}(X\times X, \rat)$.

It is possible to endow each term on the r.h.s.  of (\ref{vfgtr}) with a natural pure Hodge structure and then to choose
an isomorphism $\phi$ (\ref{vfgtr}) that is an isomorphism of pure Hodge structures. This implies
that $\pi_{\phi}$ is rational and that is has {$(p,q)$-type} $(\dim{X}, \dim{X})$, i.e. it is a Hodge class.

According to the Hodge conjecture, $\pi_\phi$ should be  algebraic (cohomology class of an algebraic cycle in $X\times X$).

\begin{??}\la{modt} {\rm (\textbf{Motivic decomposition theorem})}
Can we chose $\phi$ so that the resulting projectors $\pi_\phi$ are given by algebraic cycles?
\end{??}

The answer is positive for semismall maps (Exercise~\ref{iths}).
We have no idea if/why this should be true. We can prove something much weaker: the projectors
are absolute Hodge (in the sense of Deligne), even  motivated (in the sense of Andr\'e); see \ci{motivated}.

If it were true, then, by  applying this to the blowing up of the projective cone
over an embedded projective manifold, it would imply the Grothendieck standard conjecture of Lefschetz type
(the  inverse to the Hard Lefschetz isomorphisms are induced by algebraic cycles in the product); see the introduction to 
\ci{co-ha}.

\subsection{Exercises for Lecture 5}\la{exlz4}

\begin{exe}\la{ght} {\rm (\textbf{Restricting  perverse sheaves to general linear sections})} {\rm  Let $Y$ be quasi  projective and $P\in P(Y)$. Show
that if $i: Y_k \to Y$ is a codimension~$k$ general linear section of $Y$ relative to any fixed  embedding in
$Y \to \pn{N}$, then $P_{|Y_k}[-k]   \in P(Y_k).$ 
To do so first verify directly that  $P_{|Y_k}[-k]$ satisfies the conditions of support (use the Bertini theorem
to cut down the supports of cohomology sheaves). It not so trivial to verify the conditions of co-support, i.e. the conditions of 
 support for  the dual of $P_{|Y_k}[-k]$. Here, we simply say that we can choose $Y_k$ general, depending on $P$, so that $i^! = i^*[-2k]$
 and that  the desired conditions follow formally from this and from the duality exchange property: $(i^* (P[-k]))^\vee = i^! P^\vee [k] = i^* P^\vee [-k].$ (See \ci{decmightam}, for example).
}
\end{exe}

\begin{exe}\la{ght2} {\rm (\textbf{Restriction to general linear sections maps in cohomology})} {\rm Let $Y$ be quasi projective,  let  $P \in D(Y)$ and let $Y_k$ be the complete intersection of $k$ general linear  sections of $Y$ relative to any embedding
of $Y$ in some projective  space. Use the Lefschetz hyperplane theorem to show that
the restriction maps  $H^*(Y,P) \to H^*(Y_k, P_{|Y_k})$ are  injective
for every $* \leq -k.$ Assume in addition that $Y$ is affine. Use  the Artin vanishing theorem
for perverse sheaves  and Exercise~\ref{ght}  to show that the same restriction maps
are zero for $* >-k.$
}
\end{exe}

\begin{exe}\la{mui} {\rm (\textbf{Geometric description of perverse Leray})} {\rm 
Write out  explicitly the conclusion of Theorem~\ref{tmgdpf} in the case when $f$ is the blow-up
at the vertex 
of the affine cone  over $\pn{1}\times \pn{1}$ and verify it.}
\end{exe}

\begin{exe}\la{joau} 
{\rm 
(\textbf{Jouanolou's  trick}) Let $Y$ be quasi projective. There is a map $p: \m{Y} \to Y$ with $\m{Y}$ affine and that is a Zariski locally trivial
bundle with fiber affine spaces $\bb{A}^m$. There are various different ways to do this; following is a series of hints to reach this  goal.
 Choose a projective completion $Y\subseteq \ov{Y}$ so that the embedding is affine (e.g. with boundary a divisor).
 Argue that we may assume that $Y$ is projective. Pick a  suitably  positive rank $\dim{Y}+1$ vector bundle $E$ on $Y$,
 where suitably positive :=  $\m{O}_{\bb{P}(E)}(1)$ is very ample and has a section $s$ that is not identically zero on any projective
 fiber.  Show that taking $\pn{1}\times \pn{1} \setminus \Delta \to \pn{1}$ yields a special case of the construction above.
 Show that ${\m{Y}:= \bb{P}(E) \setminus (s=0)}$ does the job.  The usefulness of this trick in our situation is that
 if we start with ${f:X \to Y}$, with $Y$ quasi projective, we can base change to ${g: \m{X} \to \m{Y}}$
 so that now the target is affine 
 and the properties of $p$ (smooth map with ``contractible'' fibers) allow us   to prove the assertions
on $f$ by first proving them for $g$ and then ``descending'' them to~$f$.
}
\end{exe}

\begin{exe}\la{comp} {\rm 
(\textbf{Leray and perverse Leray})  In this exercise use the following: if  $j: S^o \to S$ is an open embedding of nonsingular curves and $L$ is a locally constant sheaf on $S^o,$ then $\m{IC}_S(L)= R^0j_*L.$  Let $p: \pn{1} \times \pn{1} \to \pn{1}$ be a projection and let ${b: X \to \pn{1}\times \pn{1}}$ be the blowing up at a point. Let $f:= p\circ b$. Determine and compare the perverse Leray  and the Leray filtrations for $f$ on $H^*(X,\rat)$. (Renumber the
perverse Leray one so that $1 \in \ms{P}_0 \setminus \ms{P}_{-1}$; this way they both ``start'' at the same ``time''.) 
Do the same thing, but for a Lefschetz pencil of plane curves. Note how the graded spaces for the two filtrations differ in the ``middle''.}
\end{exe}

\begin{exe}\la{mdpure}
{\rm 
(\textbf{$H^*(M_D)$ is pure}). There is a natural $\comp^*$-action on $M_D$ obtained by multiplying
the Higgs field $\phi$ by a scalar. Show that the Hitchin map is  equivariant for this action. Use the $\comp^*$-action and the properness
of the Hitchin map to show that the  closed embedding of the fiber $h^{-1}(0) \to M_D$ induces
an isomorphism in cohomology.  Use the weight inequalities listed at the end of our quick review of mixed Hodge theory
in \S\ref{exlz1}, to deduce the $H^j(M_D)$ is pure of weight $j$ for every~$j$.}
\end{exe}

\begin{exe}\la{normperv}
{\rm 
(\textbf{ Normalizing the perverse filtration}) Assume that 
\[C=\oplus_b P_b[-b] \in D(Y),   \quad    \mbox{with $P_b\in P(Y)$}. \]
Show that $\ms{P}_b H^*(Y,C[m]))= \ms{P}_{b+m} H^{*+m}(Y,C)$ (in fact, this is true in general, i.e. without
assuming that $C$ splits). Deduce that if $H^*(Y,C)\neq 0,$ then   $\exists !\, m \in \zed$ such that 
$Gr^{\ms{P}}_{b<0} H^*(Y,C[m])=0$ and $Gr^{\ms{P}}_{0} H^*(Y,C[m])\neq 0.$
The Hitchin map is of pure relative dimension $a$ with a nonsingular domain; use these two facts, and the decomposition theorem, to deduce that 
if $C= Rf_* \rat_{M_D}$, then the $m=a,$ i.e. the perverse Leray filtration
for $H^*(A, Rf_* \rat_{M_D} [a])$ ``starts'' at level zero, and that it ``ends'' at level $2a$ (trivial graded pieces after that).
Reality check: verify that a class lives in $\ms{P}_{D,2} H^4(M_D,\rat)$ iff it restricts to zero over a general line
$\Lambda^1 \subseteq A^a$.
}
\end{exe}

\begin{exe}\la{pfnmu}
{\rm 
(\textbf{The perverse filtration is not multiplicative in general}) 
First, let ${X'= S\times C}$ (surface times curve, both projective and nonsingular), then let $X$ be the blowing up
of a point in $X'$  and  let $f: X \to C$ be the natural map (blow-down followed by projection.
This is a flat map of relative dimension $2$. Let $\ms{P}$ be the perverse Leray filtration 
on $H^*(X,\rat) = H^{*-2}(C, Rf_* \rat_{X}[2])$. Verify that it lives in the interval $[0,4].$
Verify that the class $e$ of the exceptional divisor lies in $\ms{P}_{1} \setminus \ms{P}_0$
and that $e^2$ lives in $\ms{P}_{3}\setminus \ms{P}_2.$ Deduce that the perverse filtration is not multiplicative
in general (multiplicative:= $\ms{P}_i \cup \ms{P}_j \to \ms{P}_{i+j}$).
}
\end{exe}

\begin{exe}\la{motok}
{\rm 
(\textbf{When Question~\ref{modt} has an easy answer}) 
List some classes of proper maps $f:X \to Y$ such that Question~\ref{modt} has an affirmative answer.
}
\end{exe}

\textbf{Acknowledgments.} I would like to thank T. Feng, M. Goresky, L. Migliorini and R. Virk for their very useful comments.

%
%
%
%
%
%
%
\bibspread

\end{document}